\documentclass{llncs}

\usepackage{amssymb,stmaryrd,mathrsfs,dsfont,mathtools}

\usepackage{enumitem}
\usepackage[ruled,linesnumbered]{algorithm2e}
\usepackage{algpseudocode}
\usepackage{tikz}
\tikzstyle{vertex}=[circle, draw, inner sep=0pt, minimum size=2.5pt, fill=black]
\newcommand{\vertex}{\node[vertex]}
 
\usetikzlibrary{decorations.markings}
\usetikzlibrary{decorations.pathmorphing}

\begin{document}

\frontmatter         

\pagestyle{headings}  

\mainmatter             

\title{Word-Representation of Melon Graphs}

\titlerunning{Word-Representation of Melon Graphs}

\author{Khyodeno Mozhui \and K. V. Krishna}

\authorrunning{Khyodeno Mozhui \and K. V. Krishna}

\institute{Department of Mathematics\\ Indian Institute of Technology Guwahati, India\\
	\email{k.mozhui@iitg.ac.in};
	\email{kvk@iitg.ac.in}}

\maketitle         

\begin{abstract}
	The notion of word-representable graphs is a generalization of comparability graphs, in which graphs are represented by words. The complexity of word-representation of a word-representable graph is captured through representation number, whereas the corresponding concept is the permutation-representation number for comparability graphs. The graphs with the (permutation-)representation number at most two were characterized in the literature. While certain examples in the class of graphs with the (permutation-)representation number three are known, no characterization for these classes are available. In this work, we prove that the representation number of melon graphs is at most three. Further, we characterize the class of melon graphs restricted to comparability graphs and show that their permutation-representation number is also at most three. Moreover, this work characterizes the word-representable line graphs of melon graphs and establish that their representation number is at most three.
\end{abstract} 

\keywords{Word-representable graphs \and melon graphs \and line graphs \and representation number \and permutation-representation number.}

\section{Introduction and Preliminaries}
\label{intro}
Let $X$ be a finite set. A word over $X$ is a finite sequence of elements of $X$ written by juxtaposing them. The empty sequence is called the empty word and it is denoted by $\varepsilon$. A word $u$ is called a subword of a word $w$, denoted by $u \ll w$, if  $u$ is a subsequence of $w$.  If $u$ is a sequence of consecutive letters of $w$, then $u$ is called a factor. Suppose $w$ is a word over $X$ and $a, b \in X$. We say $a$ and $b$ alternate in $w$ if the subword of $w$ consisting of all occurrences of $a$ and $b$ in $w$ is of the form $abababa\cdots$ or $bababa\cdots$ of even or odd length. A word in which every letter appears exactly $k$ number of times is called a $k$-uniform word. 

A simple graph $G = (V, E)$ is called a \textit{word-representable graph} if there exists a word $w$ over its vertex set $V$ such that, for all $a,b \in V$, $\{a, b\} \in E$ if and only if $a$
and $b$ alternate in $w$. Not all graphs are word-representable and recognizing whether a graph is word-representable or not is NP-complete. The class of word-representable graphs includes several important classes of graphs, such as 3-colorable graphs, comparability graphs, and circle graphs. In \cite{kitaev08}, it was proved that if a graph is word-representable, then there are infinitely many words representing it.  Further, every word-representable graph is represented by a $k$-uniform word, for some positive integer $k$. The topic of word-representable graphs received the attention of many authors, and the literature has several contributions to the topic. For a detailed introduction to the theory of word-representable graphs, and for the notions which are used but not defined in this paper, one may refer to the monograph \cite{kitaev15mono}.

A word-representable graph is said to be $k$-word-representable if it is represented by a $k$-uniform word. The smallest $k$ such that a graph $G$ is $k$-word-representable is called the \textit{representation number} of the graph, and it is denoted by $\mathcal{R}(G)$. 
A permutationally representable graph is a word-representable graph that can be represented by a word of the form $p_1p_2\cdots p_k$ where each $p_i$ is a permutation of its vertices; in this case, the graph is called a \textit{permutationally $k$-representable graph}. It was shown in \cite{kitaev08order} that the class of permutationally representable graphs is precisely the class of comparability graphs -- the graphs that admit transitive orientation. A transitive orientation of a graph is an assignment of direction to each edge so that the adjacency relation on the vertices of the resulting directed graph is transitive. Thus, a comparability graph induces a partially ordered set (poset) based on a transitive orientation. The \textit{permutation-representation number} (in short, \textit{prn}) of a comparability graph $G$, denoted by $\mathcal{R}^p(G)$, is the smallest number $k$ such that the graph is permutationally $k$-representable. It is evident that, for a comparability graph $G$, $\mathcal{R}(G) \le \mathcal{R}^p(G)$. It can be observed that $\mathcal{R}(G) = \mathcal{R}^p(G) = 1$ if and only if $G = K_n$, where $K_n$ is the complete graph on $n$ vertices.

The class of graphs with representation number at most two is characterized as circle graphs \cite{halldorsson11}, while it is the class of permutation graphs with the \textit{prn} at most two \cite{Baker_1972}. However, there is no characterization available for graphs with representation number or \textit{prn} three. In general, it is NP-hard to determine the  representation number of a word-representable graph \cite{kitaev15mono}. It is known that the \textit{prn} of a comparability graph coincides with the dimension of its induced poset (cf. \cite{Mozhui_Krishna_2025}). Thus, for $k \ge 3$, it is NP-complete to decide whether the \textit{prn} of a comparability graph is $k$  \cite{yanna82}. Nevertheless, the representation number for some specific classes of graphs was obtained, in addition to some isolated examples. For example, it was established that the Petersen graph \cite{kitaev15mono}, prisms \cite{kitaev13}, Cartesian product of complete graphs $K_m$ and $K_2$ \cite{broere_2019}, and stacked book graphs \cite{KM_KVK_ric} have representation number at most three. 

Contributing to the class of graphs with (permutation-)representation number three, in this paper, we focus on the class of melon graphs, which are also known as generalized $\theta$-graphs or theta graphs. A \textit{melon graph} is a collection of vertex-disjoint paths with two common endpoints. We call these paths as the constituent paths of the melon graph. Following the notation given in \cite{dissaux}, a melon graph $M$ is denoted by $M = (E_1, E_2, \ldots, E_m)$, where $E_i$'s ($1 \le i \le m$) are constituent paths with common endpoints $0$ and $0'$. We consider simple graphs, and hence, there will be at most one constituent path of length one in a melon graph.
Since 1978, the melon graphs and its special classes were studied in many contexts, including chromatic uniqueness \cite{Loerinc_1978}, edge bandwidth \cite{Dennis_2000}, Tur\'{a}n number \cite{Liu_2023,Zhai_2021}, treelength \cite{dissaux}, burning number \cite{Liu_zhang_2019}, sum choice number \cite{Brause_2017,Carraher_2015}. While these graph problems are NP-hard in general, these problems were investigated for the class of melon graphs. In Section \ref{rno_melon}, we observe that the melon graphs are word-representable and prove that the representation number of a melon graph is at most three.

From the definition, it is easy to note that melon graphs are planar graphs. In \cite{Felsner_2015}, it was shown that posets induced by planar  comparability graphs have dimension at most four. Further, the posets induced by melon graphs restricted to comparability graphs are planar posets\footnote{A poset is said to be	planar if its Hasse diagram  can be drawn without edge crossings in the plane.} (see Appendix \ref{planar_poset_melon}). The problem of determining and bounding the dimension of planar posets plays a central role in the theory of posets  (cf.\cite{trotterbook}). While general bounds on the dimension of a poset in terms of its height or the number minimal elements are known (see \cite{Joret_2017,Trotter_2016}), various authors have contributed in  determining the dimension for special classes of planar posets, e.g., see \cite{biro_2021,trotter_dim}. In Section \ref{char_prn}, we characterize the class of comparability graphs within melon graphs and show that their \textit{prn} (i.e., the dimension of the corresponding planar posets) is at most three.

In \cite{kitaev_linegraphs}, it was shown that the line graph of a word-representable graph need not be word-representable, and posed an open problem to check whether there is a finite classification of forbidden subgraphs in a graph such that its line graph is word-representable. In Section \ref{line_melon}, we prove that the line graph of a melon graph is word-representable if and only if the triangular book graph with three pages is a forbidden induced subgraph. Subsequently, we show that the representation number of the line graph of a melon graph is at most three. We also establish the \textit{prn} of comparability graphs within this class.

\section{Representation Number}
\label{rno_melon}

It is evident that the complete graphs $K_2$ and $K_3$ are the only melon graphs with representation number one.  In this section, we show that the melon graphs are 3-word-representable (Theorem \ref{3-word-fg}) and classify the graphs with representation numbers two and three (see Theorem \ref{3-induced-fg}). 
In this work, we use the following theorem.

\begin{theorem}[\cite{kitaev08}] \label{subdiv}
	Let $G =(V,E)$ be a $3$-word-representable graph and $a, b \in V$. If $H$ is the graph obtained from $G$ by adding a path of length at least three connecting $a$ and $b$, then $H$ is $3$-word-representable. 
\end{theorem}

\begin{theorem}\label{3-word-fg}
	Every melon graph is 3-word-representable. 
\end{theorem}

\begin{proof} 
	Suppose all the constituent paths of a melon graph $M = (E_1, \ldots, E_m)$ are of length at most two. Since $M$ is simple, there will be at most one path of length one. If $M$ has a path of length one, say $E_1$, then all other paths $E_i$ ($2 \le i \le m$) are of length two, say $E_i = \langle 0, x_i, 0' \rangle $, with intermediate vertex $x_i$. Clearly, the 2-uniform word $x_2 x_3 \ldots x_m 0 0' x_m \ldots x_3 x_2 0 0'$ represents $M$. If all paths are of length two in $M$, then $M = K_{2,m}$, a complete bipartite graph, so it is a circle graph. Note that a path of length at least three is also a circle graph.
	
	Further, it can be observed that an arbitrary melon graph can be obtained by adding paths of length at least three to a circle graph described in the previous paragraph, i.e., a path of length at least three or a melon graph in which all paths are of length at most two. Since adding a path of length at least three to a 3-word-representable graph results in a 3-word-representable graph (cf. Theorem \ref{subdiv}), we have all melon graphs are 3-word-representable. \qed  
\end{proof}	

We now classify the melon graphs based on their number of constituent paths and give their representation numbers through a series of lemmas.

Let $M = (E_1, E_2, \ldots, E_m)$ be a melon graph. If $m \le 2$, then $M$ is either a path or a cycle; in which case, if $M \ne K_2$ or $K_3$, then $\mathcal{R}(M) = 2$. 

\begin{lemma}
	For $m \ge 3$, if $m-2$ constituent paths of $M$ are of length at most two, then $\mathcal{R}(M) = 2$.   
\end{lemma}	
\begin{proof}
	In Theorem \ref{3-word-fg}, we observed that if the length of each constituent path is at most two, then $\mathcal{R}(M) = 2$. If $M$ has one or two constituent paths of length at least three, we show the result as per the following cases:   
	\begin{itemize}
		\item  Suppose $M$ has a constituent path, say $E_1$, of length at least three and others of length two. Let $E_1 = \langle 0, a_1, a_2, \ldots, a_{k}, 0' \rangle$, and $E_i = \langle 0, x_i, 0' \rangle$, for all $2 \le i \le m$. Consider the cycle $(E_1,E_2)$ and construct a 2-uniform word, as per the following, representing the cycle using the technique given in \cite[Section 3.2]{kitaev15mono}:
		
		$$w = 0 x_2 a_1 0 u 0' a_{k} x_2 0',$$ where $u = a_2 a_1 \cdots  a_{i} a_{i-1} \cdots a_{k} a_{k-1}$ (for clarity, $2 \le i \le k$).
		Since $x_2, x_3, \ldots, x_m$ have the same neighbourhood, i.e., $\{0,0'\}$, replace first occurrence of $x_2$ by the word $x_2 x_3 \cdots x_m$, and the second occurrence of $x_2$ by the word  $x_m  \cdots x_3 x_2$ in $w$ to get the following 2-uniform word, which evidently represents $M$:
		$$0 x_2  \cdots x_i \cdots x_m a_1 0 u 0' a_{k} x_m \cdots x_i \cdots x_2 0',$$ (for clarity, $2 \le i \le m$).
		
		\item Suppose $M$ has two constituent paths, say $E_1$ and $E_2$, of length at least three and others of length two. For $k_1, k_2 \ge 2$, let $E_1 = \langle 0, a_1, a_2, \ldots, a_{k_1}, 0' \rangle$,  $E_2 = \langle 0, b_1, b_2, \ldots, b_{k_2}, 0' \rangle$, and $E_i = \langle 0, x_i, 0' \rangle$, for all $3 \le i \le m$. We construct a 2-uniform word $w'$, representing the cycle $(E_1,E_2)$, given by 
		$$w' = 0 b_1 a_1 0 u_1 0' a_{k_1}  b_{k_2} 0' v_2,$$ where $u_1 = a_2 a_1 \cdots a_{i} a_{i-1} \cdots a_{k_1} a_{k_1-1}$  and $v_2 = b_{k_2 - 1} b_{k_2} \cdots b_{j-1} b_{j}\cdots b_1 b_2$ (for clarity, $2 \le i \le k_1$ and $2 \le j \le k_2$).			
		We substitute the factors $0 b_1 a_1 0$ and $0' a_{k_1} b_{k_2} 0'$ of $w'$ by $0 b_1 x_3 \cdots x_i \cdots x_m a_1 0$ and $0' a_{k_1} x_m \cdots x_i \cdots x_3 b_{k_2} 0'$, respectively, to get the following 2-uniform word, which represents $M$:
		$$0 b_1  x_3 \cdots x_i \cdots x_m  a_1 0  u_1   0' a_{k_1}  x_m \cdots x_i \cdots x_3  b_{k_2} 0' v_2,$$ (for clarity, $3 \le i \le m$).	
		
		\item Suppose $M$ has a constituent path, say $E_1$, of length at least three, a constituent path, say $E_2$, of length one, and others of length two. Let $E_1 = \langle 0, a_1, a_2, \ldots, a_{k_1}, 0' \rangle$, $E_2 = \langle 0, 0' \rangle$, and $E_i = \langle 0,x_i,0' \rangle$, for all $3 \le i \le m$. Consider the cycle $(E_1, E_2)$ and construct the following 2-uniform word representing the cycle:
		$$w'' =  0 0' a_1 0 u_1 0' a_{k_1}.$$
		Since $x_3,x_4 \ldots, x_m$ have the same neighbourhood, i.e., $\{0,0'\}$, replace the factor $00'$ by the word $x_3 \cdots x_i \cdots x_m 0 0' x_m \cdots x_i \cdots x_3$ to get the following 2-uniform word, which represents $M$:
		$$x_3 \cdots x_i \cdots x_m 0 0' x_m \cdots x_i \cdots x_3 a_1 0 u_1 0' a_{k_1},$$ (for clarity, $3 \le i \le m$).
		
		\item Suppose $M$ has two constituent paths, say $E_1$ and $E_2$, of length at least three, a constituent path, say $E_3$, of length one, and others of length two. Let $E_1 = \langle 0, a_1, a_2, \ldots, a_{k_1}, 0' \rangle$,  $E_2 = \langle 0, b_1, b_2, \ldots, b_{k_2}, 0' \rangle$, $E_3 = \langle 0, 0'\rangle$, and $E_i = \langle 0, x_i, 0' \rangle$, for all $4 \le i \le m$. Then, as per the previous case, consider the 2-uniform word representing $(E_1, E_3, E_4, \ldots, E_m)$ as per the following:
		$$x_4 \cdots x_i \cdots x_m 0 0' x_m \cdots x_i \cdots x_4  a_1 0 u_1 0' a_{k_1},$$ (for clarity, $4 \le i \le m$).
		In this word, replace the factor $00'$ by the word $b_1 0 v_2^R 0' b_{k_2}$, where $v_2^R$ is the reversal\footnote{The reversal of a word is the word obtained by writing its letters in the reverse order.} of $v_2$, to get the following 2-uniform word representing $M$:
		$$x_4 \cdots x_i \cdots x_m b_1 0 v_2^R 0' b_{k_2} x_m \cdots x_i \cdots x_4 a_1 0 u_1 0' a_{k_1}.$$ 
	\end{itemize}
	Hence, $\mathcal{R}(M) = 2$. \qed
\end{proof}

\begin{figure}
	\centering
	\begin{minipage}{.4\textwidth}
		\centering
		\begin{tikzpicture}[scale=0.9]
			\vertex (1) at (0,3) [label=above:$0'$] {};  
			\vertex (2) at (-1.5,2) [label=left:$1$] {}; 
			\vertex (3) at (0,2) [label=left:$2$] {}; 
			\vertex (4) at (1.5,2) [label=right:$3$] {}; 
			\vertex (5) at (0,0) [label=below:$0$] {}; 
			\vertex (6) at (-1.5,1) [label=left:$4$] {}; 
			\vertex (7) at (0,1) [label=left:$5$] {}; 
			\vertex (8) at (1.5,1) [label=right:$6$] {}; 
			\path
			(1) edge (2)
			(1) edge (3)
			(1) edge (4)
			(2) edge (6)
			(3) edge (7)
			(4) edge (8)
			(1) edge (2)
			(1) edge (3)
			(1) edge (4)
			(2) edge (6)
			(3) edge (7)
			(4) edge (8)
			(5) edge (6)
			(5) edge(7)
			(5) edge (8);
		\end{tikzpicture}		
		\caption{$M_3$}
		\label{F3}
	\end{minipage}
	\begin{minipage}{.5\textwidth}
		\centering
		\begin{tikzpicture}[scale=0.7]
			\vertex (1) at (0,4) [label=above:$0'$] {};  
			\vertex (2) at (-1.5,3) [label=left:$1$] {}; 
			\vertex (3) at (0,2) [label=left:$2$] {}; 
			\vertex (4) at (1.5,3) [label=right:$3$] {}; 
			\vertex (5) at (0,0) [label=below:$0$] {}; 
			\vertex (6) at (-1.5,1) [label=left:$4$] {}; 
			\vertex (7) at (0,1) [label=left:$5$] {}; 
			\vertex (8) at (1.5,1) [label=right:$6$] {}; 
			\path		
			(2) edge (6)
			(3) edge (7)
			(4) edge (8)
			(2) edge (6)
			(3) edge (7)
			(4) edge (8)
			(5) edge (6)
			(5) edge(7)
			(5) edge (8)
			(2) edge (4)
			(2) edge (3)
			(3) edge (4);
			\path[dashed]
			(1) edge (2)
			(1) edge (3)
			(1) edge (4);
		\end{tikzpicture}
		\caption{Local complementation of $M_3$ at 0'}
		\label{localf3}
	\end{minipage}		
\end{figure}

\begin{lemma} \label{f_3 ref}
	Let $M_3$ be a melon graph with three constituent paths, each of length exactly three. Then $\mathcal{R}(M_3) = 3$.
\end{lemma}
\begin{proof}
	Assume $\mathcal{R}(M_3) = 2$. Since the class of circle graphs is closed under local complementation\footnote{The local complement of a graph at vertex $v$ is the graph obtained by complementing the edges of the induced subgraph with the vertex set $N(v)$, the neighborhood of $v$.} (cf. \cite{bouchet}), consider the graph as shown in Fig. \ref{localf3}, obtained by local complementation at vertex $0'$ of $M_3$ (given in Fig. \ref{F3}). Note that its induced subgraph represented in solid edges has representation number three (compare with the list of graphs in \cite[Figure 4]{Akgun_2019}).  Hence, the graph in Fig. \ref{localf3} is not a circle graph and so is not $M_3$. Hence, $\mathcal{R}(M_3) = 3$. \qed
\end{proof}

Let $\mathcal{M}_3$ be the family of melon graphs with three constituent paths, each of length at least three.

\begin{lemma}\label{vertex_minor}
	If $M \in \mathcal{M}_3$, then $\mathcal{R}(M) = 3$.
\end{lemma}
\begin{proof}
	Recall that the class of circle graphs is closed under vertex-minor\footnote{A graph obtained by a sequence of vertex deletions and local complementations in a graph is called its vertex-minor.} \cite{bouchet}. Note that $M_3$ can be obtained from $M$ by a successive operations of local complementation followed by vertex deletion on the intermediate vertices of the constituent paths. Thus, if $\mathcal{R}(M) = 2$, then $M_3$ must be a circle graph; a contradiction to Lemma \ref{f_3 ref}. Hence,  $\mathcal{R}(M) = 3$. \qed
\end{proof}

\begin{figure}[t]
	\centering
	\begin{minipage}{.4\textwidth}
		\centering
		\begin{tikzpicture}[scale=0.7]
			\vertex (1) at (0,3.3) [label=above:$0'$] {};  
			\vertex (2) at (-1.5,2) [label=left:$1$] {}; 
			\vertex (3) at (1,2) [label=left:$2$] {}; 
			\vertex (4) at (1.5,2) [label=right:$3$] {}; 
			\vertex (5) at (0,0) [label=below:$0$] {}; 
			\vertex (6) at (-1.5,1) [label=left:$4$] {}; 
			\vertex (7) at (1,1) [label=left:$5$] {}; 
			\vertex (8) at (1.5,1) [label=right:$6$] {}; 
			\path
			(1) edge (2)
			(1) edge (3)
			(1) edge (4)
			(1) edge (5)
			(2) edge (6)
			(3) edge (7)
			(4) edge (8)
			(1) edge (2)
			(1) edge (3)
			(1) edge (4)
			(1) edge (5)
			(2) edge (6)
			(3) edge (7)
			(4) edge (8)
			(5) edge (6)
			(5) edge(7)
			(5) edge (8);
		\end{tikzpicture}
		\caption{$B_3$}
		\label{F3'}
	\end{minipage}
	\begin{minipage}{.5\textwidth}
		\centering
		\begin{tikzpicture}[scale=0.7]
			\vertex (1) at (0,0) [label= left:$0$] {};  
			\vertex (2) at (-2.2,-1.5) [label=below:$1$] {}; 
			\vertex (3) at (2.2,-1.5) [label=below:$2$] {}; 
			\vertex (4) at (0,1.8) [label=above:$3$] {}; 
			
			\vertex (5) at (-0.5,0.3) [label= left:$6$] {}; 
			\vertex (6) at (0.7,0) [label=above:$5$] {}; 
			\vertex (7) at (0,-0.7) [label=below:$4$] {}; 
			
			\vertex (0) at (3, -1) [label= right:$0'$] {}; 
			\path
			(1) edge (2)
			(1) edge (3)
			(1) edge (4)
			(1) edge (5)
			(1) edge (6)
			(1) edge (7)
			(2) edge (7)
			(2) edge (3)
			(4) edge (2)
			(4) edge (5)
			(3) edge (6)
			(3) edge (4);
			\path[dashed]
			(4) edge (0)
			(3) edge (0)
			(2) edge (0)
			(1) edge (0);
		\end{tikzpicture}
		\caption{Local complementation of $B_3$ at 0'}
		\label{localf3'}
	\end{minipage} 	
\end{figure}

A \textit{book graph} is a melon graph in which $0$ and $0'$ are adjacent and every other constituent path is of length three, called a page. A book graph with $m$ pages is denoted by $B_m$. 

\begin{lemma}\label{b_3-rep}
	The representation number of the book graph $B_3$ (given in Fig. \ref{F3'}) is three. 
\end{lemma} 
\begin{proof}
	Suppose $\mathcal{R}(B_3) = 2$. Since the circle graphs are closed under local complementation, consider the graph obtained by local complementation at vertex $0'$ of $B_3$, as shown in Fig. \ref{localf3'}. Note that its induced subgraph represented in solid edges in Fig. \ref{localf3'} is not a word-representable graph (compare with the list of graphs in \cite[Figure 3.9]{kitaev15mono}).  Hence, the graph in Fig. \ref{localf3'} is not a circle graph and so is not $B_3$. Therefore, $\mathcal{R}(B_3) \ge 3$. Since $\mathcal{R}(B_3) \le 3$, we have $\mathcal{R}(B_3) = 3$. \qed
\end{proof}

We have the following consequence of Lemma \ref{b_3-rep}. 
Let $\mathcal{M}_4$ be the family of melon graphs with four constituent paths, of which one is of length one and the others of length at least three. 

\begin{corollary}\label{vertex_minor_2}
	If $M \in \mathcal{M}_4$, then $\mathcal{R}(M) = 3$.
\end{corollary} 
\begin{proof}
	Note that $B_3$ is a vertex-minor of $M$. Hence,  $\mathcal{R}(M) = 3$. \qed	
\end{proof}

We now summarize the above results and give two characterizations for the representation number of melon graphs. Indeed, we show that the representation number of a melon graph is three if and only if it has three constituent paths, each of length at least three. 

Let $M = (E_1, \ldots, E_m)$ be a melon graph. If the number of constituent paths of length at least three is at most two, then by Lemma 4, $\mathcal{R}(M) = 2$. On the other hand, $M$ has at least three constituent paths of length at least three. Suppose $E_1, E_2$ and $E_3$ are of length at least three. If $M$ has a constituent path of length one (i.e., $0$ and $0'$ are adjacent), say $E_4$, then it can be observed that $M$ has an induced subgraph from $\mathcal{M}_4$, which is obtained by deleting $E_5, \ldots, E_m$ except the endpoints. Hence, by Corollary \ref{vertex_minor_2}, $\mathcal{R}(M) = 3$. If $0$ and $0'$ are not adjacent in $M$, then we can observe that $M$ has an induced subgraph from $\mathcal{M}_3$ so that $\mathcal{R}(M) = 3$. Thus, we have the following characterization of melon graphs in terms of induced subgraphs with respect to the representation number. 

\begin{theorem}\label{3-induced-fg}
	Let $M$ be a melon graph. Then, $\mathcal{R}(M)= 3$ if and only if it contains an induced subgraph from $\mathcal{M}_3 \cup \mathcal{M}_4$.
\end{theorem}

Note that the melon graph $M_3$ is a vertex-minor of every member in $\mathcal{M}_3$. Also, the melon graphs in $\mathcal{M}_4$ contain $B_3$ as vertex-minor. Thus, we have the following corollary with a characterization in terms of vertex-minors. 

\begin{corollary}
	Let $M$ be a melon graph. Then, $\mathcal{R}(M)= 3$ if and only if $M$ contains $M_3$ or $B_3$ as a vertex-minor.
\end{corollary}

Further, we have the following corollary for the representation number of book graphs. 

\begin{corollary}
	Let $B_m$ be a book graph. Then, $\mathcal{R}(B_m)= 3$ if and only if it contains $B_3$ as an induced subgraph.
\end{corollary}

\section{Characterization for Comparability and the \textit{prn}}
\label{char_prn}

All melon graphs are word-representable but not all melon graphs are comparability graphs. For example, if we consider a melon graph with two constituent paths, one of even length and the other of odd length, then the melon graph is an odd cycle, which is not a comparability graph, except $K_3$. In this section, we characterize the melon graphs which admit transitive orientations, i.e., which are comparability graphs. Further, we determine their permutation-representation number.

First, we prove that the melon graphs  with constituent paths of the same parity (either all of them are of even length or of odd length) are comparability graphs.

\begin{lemma}
	If all the constituent paths of a melon graph $M$ have the same parity, then $M$ is a bipartite graph. Hence, $M$ is a comparability graph.
\end{lemma}

\begin{proof}
	We prove the result using induction on the number of constituent paths $m$ in the melon graph $M = (E_1, \ldots, E_m)$. 
	
	If $m = 1$ or 2, then clearly $M$ is a bipartite graph, as $M$ is a path or an even cycle, respectively. Consider a melon graph $M$ with $k+1$ constituent paths of the same parity. By inductive hypothesis, the melon graph $(E_1, \ldots, E_k)$ is a bipartite graph with bipartition, say $\{A, B\}$, of vertex set. Without loss of generality, suppose $0 \in A$. Let $E_{k+1} = \langle a_0 = 0, a_1, \ldots, a_r = 0'\rangle$. Place the vertices of $E_{k+1}$ such that $a_i \in A$ if $i$ is even and $a_i \in B$ if $i$ is odd. If all the constituent paths of $M$ are of even length, then clearly $0' \in A$; otherwise, $0' \in B$. In any case, the proposed partition of vertices of $E_{k+1}$ ensures the bipartition so that $M$ is a bipartite graph.\qed
\end{proof}

Suppose a melon graph $M$ has constituent paths of lengths both even and odd. If $M$ has a constituent path of even length at least four, then this along with any odd length constituent path results in an odd cycle of length at least five so that $M$ is not a comparability graph. 

Now, suppose all constituent paths of even parity have length exactly two. If $0$ and $0'$ are not adjacent, then $M$ has a constituent path of odd length at least three. This along with a constituent path of length two form an odd cycle of length at least five. Again, $M$ is not a comparability graph in this case. In case $0$ and $0'$ are adjacent, we observe that $M$ is a comparability graph. First consider the subgraph, say $M'$, of $M$ consisting of all odd length constituent paths. Note that $M'$ is a bipartite graph with bipartition, say $\{A, B\}$ and $0 \in A$. Suppose the edges of $M'$ are oriented from $A$ to $B$. Let $E' = \langle 0, x, 0' \rangle$ be any constituent path of length two in $M$. Orient the edges of $E'$ as $\overrightarrow{0x}$ (assign direction from $0$ to $x$) and $\overrightarrow{x0'}$ (assign direction from $x$ to $0'$) and note that $M'$ along with $E'$ is a comparability graph as $0$ and $0'$ are adjacent and the transitivity is preserved. Similarly, any constituent path of length two in $M$ can be oriented so that $M$ is a comparability graph. 

Thus, we have the following theorem characterizing the melon graphs admitting transitive orientations.

\begin{theorem}\label{chr_com_file}
	A melon graph $M = (E_1, \ldots, E_m)$ is a comparability graph if and only if one of the following holds: 
	\begin{enumerate}
		\item All constituent paths of $M$ have the same parity.
		\item The endpoints $0$ and $0'$ are adjacent, and all constituent paths of even parity are of length two. 
	\end{enumerate}	
\end{theorem}

\subsection{Permutation-Representation Number}
\label{prn_melon}

Let $M_c$ be a melon graph that satisfies one of the criteria given in Theorem \ref{chr_com_file}, i.e., $M_c$ is a comparability graph.  We show that the \textit{prn} of the melon graph $M_c$ is at most three by constructing three permutations on the vertices of $M_c$. 

First we construct permutations for paths and even cycles and extend them to the desired permutations on the vertices of $M_c$. 
It is known that paths are permutationally 2-representable and even cycles are permutationally 3-representable (see \cite{trotter_dim}). However, we construct specific permutations for paths and even cycles (in Lemmas \ref{path_word}, \ref{path_2k-1}, and \ref{even_cycle_word}), which are useful in Theorems \ref{odd_parity}, \ref{even_parity} and \ref{00'adj_fc}.   

For $k \ge 2$, let $u$ and $v$ be the words over the vertices $c_2, c_3, \cdots, c_{2k-1}$ given by 
\begin{eqnarray}
	\label{eq-1} u &= c_{4} c_{3} \cdots c_{2i} c_{2i-1} \cdots c_{2k-2} c_{2k-3},\; (\text{for clarity, } 2 \le i \le k-1);\\
	\label{eq-2} v &= c_{2k-2} c_{2k-1} \cdots c_{2j} c_{2j+1} \cdots c_2 c_3,\; (\text{for clarity, } 1 \le j \le k-1).
\end{eqnarray} 

\begin{lemma} \label{path_word}
	For $k \ge 2$, the path $\langle  c_1, c_2,  \cdots, c_{2k-1}, c_{2k} \rangle$ on $2k$ vertices is represented by the concatenation of the permutations $p = c_{2} c_1 u c_{2k} c_{2k-1}$ and $q =c_{2k}  v c_1$, where $u$ and $v$ are the words given in equations (\ref{eq-1}) and (\ref{eq-2}), respectively.
\end{lemma}

\begin{proof}
	Let $P_{2k} = \langle  c_1, c_2,  \cdots, c_{2k-1}, c_{2k} \rangle$. We prove the result by induction on $k$.  For $k =2$,  one can observe that the word $c_2c_1c_4 c_3 c_4 c_2c_3 c_1$ represents the path $P_4$ permutationally.
	Assume that the statement holds for $P_{2k}$. Consider the following two permutations on the vertices of the path $P_{2k+2}$:
	\begin{align*}
		p' &= c_2 c_1 u c_{2k} c_{2t-1} c_{2k+2} c_{2k+1}\\
		q' &= c_{2k+2}c_{2k} c_{2k+1} v c_1,	
	\end{align*}
	Removing the vertices $c_{2k+1}$ and $c_{2k+2}$ from $p'q'$ yields the word $pq$ which represents $P_{2k}$ by the induction hypothesis. It is sufficient to show that for $a \in \{c_1, c_2, \ldots, c_{2k+1}\}$ and $b \in \{c_{2k+1}, c_{2k+2}\}$, $a$ and $b$ are adjacent if and only if $a$ and $b$ alternate in $p'q'$.
	
	Suppose $a$ and $b$ are adjacent. If $b = c_{2k+1}$, then $a = c_{2k}$; and note that $abab \ll p'q'$. If $b = c_{2k+2}$, then $a = c_{2k+1} $; and $baba \ll p'q'$. Thus, $a$ and $b$ alternate in $p'q'$. Conversely, suppose $a$ and $b$ are non-adjacent. If $b = c_{2k+1}$, then $a = c_1, c_2, \ldots,$ or $c_{2k-1}$. In addtion, $a$ is also $c_{2k}$ when $b = c_{2k+2}$. In any case, we have $ab \ll p'$ and $ba \ll q'$. Thus, $a$ and $b$ do not alternate in $p'q'$.		
	Hence, the word $ p'q'$ represents the graph $P_{2k+2}$ permutationally. \qed 
\end{proof}

By removing $c_{2k}$ from $p$ and $q$ in Lemma \ref{path_word} one can observe that the resultant $pq$ represents the path $\langle  c_1, c_2,  \cdots, c_{2k-1} \rangle$ on odd number of vertices. However, for odd case, we need three permutations to represent the graph $M_c$ in Theorem \ref{even_parity}. From the permutation $p$ we construct the permutation $r = c_2u c_{2k-1} c_1$ by moving the vertex $c_1$ to the end of the permutation. It can be observed that $c_1$ still alternates with $c_2$ in the word $pqr$ and non-alternating of $c_1$ with any other vertex is taken care by the subword $pq$. Hence, we have the following lemma. 

\begin{lemma}\label{path_2k-1}
	For $k \ge 2$, the path $\langle  c_1, c_2,  \cdots, c_{2k-1}  \rangle$ on $2k-1$ vertices is represented  by the concatenation of the permutations $p = c_{2} c_1 u   c_{2k-1}$, $q = v c_1$ and $r = c_2u c_{2k-1} c_1$, where $u$ and $v$ are the words given in equations (\ref{eq-1}) and (\ref{eq-2}), respectively.
\end{lemma}

\begin{remark}\label{3-uniform_path}
	Let $p$ and $q$ be permutations on the vertex set of a graph such that $pq$ represents the graph permutationally. It is evident that the word $pqq$ also represents the graph permutationally.
\end{remark}

\begin{lemma}\label{even_cycle_word}
	For $k \ge 1$, let $\langle 0', c_1, c_2, \ldots, c_{2k}, 0 \rangle$ be a path on $2k+2$ vertices. If $0$ and $0'$ are adjacent, then the corresponding even cycle $C$ is  represented by the concatenation of the permutations $w_1 = 0' c_2 c_1 uc_{2k} c_{2k-1} 0$, $w_2 = c_{2k} v 0' c_1 0$ and $w_3 = 0' c_{2k} 0 v c_1$, where $u$ and $v$ are the words given in equations (\ref{eq-1}) and (\ref{eq-2}), respectively.
\end{lemma}

\begin{proof}
	Let $V = \{c_1,c_2,\ldots, c_{2k}\}$. Note that $C[V]$, the subgraph of $C$ induced by $V$, is the path $\langle c_1, c_2, \ldots, c_{2k} \rangle$. By Lemma \ref{path_word} and Remark \ref{3-uniform_path}, the word $pqq$ represents $C[V]$ permutationally.  It is evident that $w_1w_2w_3 |_{V} = pqq$.  It is sufficient to show that, for $a \in V \cup \{0'\}$ and $b \in \{0', 0\}$, $a$ and $b$ are adjacent if and only if $a$ and $b$ alternate in $w_1w_2w_3$. 
	
	Suppose $a$ and $b$ are adjacent. If $a = c_1$, then $b = 0'$ and $bababa \ll w_1w_2w_3$. If $a \in \{0', c_{2k}\}$, then $b = 0$ and $ababab \ll w_1w_2w_3$. Thus, $c$ and $d$ alternate in $w_1w_2w_3$.
	Conversely, suppose $a$ and $b$ are non-adjacent. If $b = 0'$, then $ba \ll w_1$ and $ab \ll w_2$. If $b = 0$, then $ab \ll w_1$ and $ba \ll w_3$. 
	Thus, $a$ and $b$ do  not alternate in $w_1w_2w_3$. Hence, the word $w_1w_2w_3$ represents the even cycle $C$ permutationally.	\qed
\end{proof}

\begin{figure}[t!]
	\centering	
	\begin{minipage}{.5\textwidth}
		\centering
		\begin{tikzpicture}[scale=0.8]
			%\node (F) at (-1,-4.5) [label=below:$M$] {};
			\vertex (0) at (0,0) [label=above:$0'$] {};  
			\vertex (1) at (-2,-1) [label=left:$a_1'$] {}; 
			\vertex (2) at (-1,-1) [label=left:$a_2'$] {}; 
			\vertex (3) at (0,-1) [label=right:$a_3'$] {}; 
			\vertex (4) at (2,-1) [label=right:$a_m'$] {}; 
			\vertex (5) at (1,-1) [label=above:$ $] {}; 
			\vertex (10) at (0,-4.25) [label=below:$0$] {};  
			\vertex (11) at (-2,-3.25) [label=left:$a_1$] {}; 
			\vertex (12) at (-1,-3.25) [label=left:$a_2$] {}; 
			\vertex (13) at (0,-3.25) [label=right:$a_3$] {}; 
			\vertex (14) at (2,-3.25) [label=right:$a_m$] {}; 
			\vertex (15) at (1,-3.25) [label=below:$ $] {}; 
			\vertex (21) at (-2,-1.75) [label=left:$ $] {}; 
			\vertex (22) at (-1,-1.75) [label=left:$ $] {}; 
			\vertex (23) at (0,-1.75) [label=right:$ $] {}; 
			\vertex (24) at (2,-1.75) [label=right:$ $] {}; 
			\vertex (25) at (1,-1.75) [label=right:$ $] {}; 
			\vertex (31) at (-2,-2.55) [label=left:$ $] {}; 
			\vertex (32) at (-1,-2.55) [label=left:$ $] {}; 
			\vertex (33) at (0,-2.55) [label=right:$ $] {}; 
			\vertex (34) at (2,-2.55) [label=right:$ $] {}; 
			\vertex (35) at (1,-2.55) [label=right:$ $] {}; 
			\path
			(0) edge (1)
			(0) edge (2)
			(0) edge (3)
			(0) edge (4) 
			(10) edge (11)
			(10) edge (12)
			(10) edge (13)
			(10) edge (14)
			(0) edge (5)
			(10) edge (15)		
			;
			\path[dashed]
			(11) edge (31)
			(12) edge (32)
			(13) edge (33)
			(14) edge (34)
			(15) edge (35)		
			(21) edge (31)
			(22) edge (32)
			(23) edge (33)
			(24) edge (34)
			(25) edge (35)
			(1) edge (21)
			(2) edge (22)
			(3) edge (23)
			(4) edge (24)
			(5) edge (25);	
		\end{tikzpicture}		
	\end{minipage}	
	\caption{Melon graph $M_c$}
	\label{file_graph}	
\end{figure}
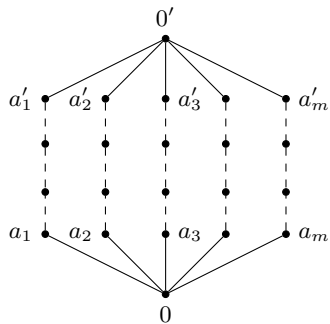

Let a melon graph $M_c = (E_1, \ldots, E_m)$ be a comparability graph.
We determine the \textit{prn} of $M_c$ in various cases in view of Theorem \ref{chr_com_file}. If $0$ and $0'$ are non-adjacent in $M_c$, we show that $\mathcal{R}^p(M_c) \le 3$ in two theorems, viz., Theorem \ref{odd_parity} and Theorem \ref{even_parity}, based on the parity of the constituent paths. In Theorem \ref{00'adj_fc}, we consider remaining cases in which $0$ and $0'$ are adjacent in $M_c$ and prove that $\mathcal{R}^p(M_c) \le 3$.  

Suppose $a_i, a_i'$ are the vertices adjacent to $0$ and $0'$, respectively, in the constituent paths $E_i$ (for $1 \le i\le m$) of $M_c$ as shown in Fig. \ref{file_graph}. 

\begin{theorem}\label{odd_parity}
	If all the constituent paths of a melon graph $M_c$ are of odd parity and of length at least three, then $\mathcal{R}^p(M_c) \le 3$. 
\end{theorem}
\begin{proof}
	Note that the endpoints  $0'$ and $0$ are non-adjacent and hence each constituent path is an induced subgraph of $M_c$. 
	Since all constituent paths are of odd parity, each $E_i$ is a path on even number of vertices. Consider the permutations $p_i = a_i' 0' u_i 0 a_i$  and $q_i = 0 v_i 0' $, where $u_i$ and $v_i$ are permutations on the remaining intermediate vertices of $E_i$ as stated in Lemma \ref{path_word}, and note that $p_iq_iq_i$ represents $E_i$ permutationally  (by Remark \ref{3-uniform_path}). 
	Using $p_i$'s and $q_i$'s, construct the following three permutations on the vertices of $M_c$: 
	\begin{align*}
		p_1' &= a_1' \cdots a_i' \cdots a_m' 0' u_m  \cdots u_i \cdots u_1 0 a_m \cdots a_i \cdots a_1\\
		p_2' &= 0 v_1 \cdots v_i \cdots v_m 0'\\
		p_3' &= 0 v_m \cdots v_i \cdots v_1 0',
	\end{align*} (for clarity, $1 \le i \le m$).
	Using induction on the number of constituent paths $m$ in $M_c$, we show that two vertices are adjacent in $M_c$ if and only if they alternate in $p_1'p_2'p_3'$.	While the statement is true for $m=1$ by Lemma \ref{path_word} and Remark \ref{3-uniform_path},  assume the statement holds if $M_c$ has at most $m - 1$ constituent paths. When $M_c$ has $m$ constituent paths, remove the intermediate vertices of $E_{m}$ from $p_1'p_2'p_3'$. By the induction hypothesis, the resulting word represents the melon graph $(E_1, E_2,\ldots, E_{m-1})$. It is also evident that $(p_1'p_2'p_3')|_{E_m}$ represents $E_m$ permutationally. It remains to show that, for $a \in E_1 \cup E_2 \cup \ldots \cup E_{m-1}$ and $b \in E_{m} \setminus \{0,0'\}$, $a$ and $b$ are adjacent in $M_c$ if and only if $a$ and $b$ alternate in $p_1'p_2'p_3'$.
	
	Suppose  $a$ and $b$ are adjacent in $M_c$. Clearly, $a \in \{0',0\}$ and $b \in \{a_{m}', a_{m}\}$. If $a = 0$, then $b = a_{m}$; in this case, we have $ab \ll p_1'$, $ab \ll 0 v_{m} \ll p_2'$, and $a b \ll 0 v_{m} \ll p_3'$. In case $a = 0'$, then $b = a_{m}'$ and $ba \ll p_1'$, $ba \ll  v_{m}0' \ll p_2'$, and $ba \ll  v_{m}0' \ll p_3'$. Hence, in any case, $a$ and $b$ alternate in $p_1'p_2'p_3'$.
	
	Conversely, suppose that $a$ and $b$ are not adjacent in $M_c$. Then, $a$ must be an intermediate vertex of some path $E_i$ ($1 \le i \le m-1$) so that $ab \ll v_i v_{m} \ll p_2'$ and $ba \ll v_{m} v_{i} \ll p_3'$. Thus, $a$ and $b$ do not alternate in $p_1'p_2'p_3'$. Hence, $p_1'p_2'p_3'$ represents $M_c$ permutationally so that $\mathcal{R}^p(M_c) \le 3$.
	\qed
\end{proof}

\begin{theorem}\label{even_parity}
	If all the constituent paths of a melon graph $M_c$ are of even parity, then $\mathcal{R}^p(M_c) \le 3$.
\end{theorem}
\begin{proof}
	Without loss of generality, let $M_c = (E_1,E_2, \ldots, E_m)$ such that $|E_1| \le |E_2|\le \cdots \le |E_m|$, where $|E_i|$ is the number of vertices of $E_i$. For $1 \le i \le m$, $E_i$ is of length two then $a_i' = a_i$ and the word $a_i 0'0 a_i 0  0' a_i 0 0' $ represents the path $E_i$. If $E_i$ is of length at least four, consider the permutations $p_i = a_i' 0' u_i 0$  and $q_i = v_i 0'$, and $r_i = a_i' u_i 0 0'$, where $u_i$ and $v_i$ are permutations on the remaining intermediate vertices of $E_i$ as stated in Lemma \ref{path_2k-1}, and note that $p_iq_ir_i$ represents $E_i$ permutationally. Let $v_i = a_i0v_i'$, as per the structure of $v_i$. Construct the following three permutations on the vertices of $M_c$: 
	\begin{align*}
		p_1' &= a_1' \cdots a_i' \cdots a_m' 0' u_1  \cdots u_i \cdots  u_m 0 \\
		p_2' &= a_1  \cdots a_i  \cdots a_m  0  v_1'  \cdots v_i' \cdots v_m' 0'\\
		p_3' &= a_m'u_m  \cdots a_i' u_i \cdots a_1' u_1'  0 0',
	\end{align*} (for clarity, $1 \le i \le m$).
	If $E_i$ if of length two, then in the permutations $p_1'$ and $p_3'$, put $a_i = \epsilon$, and in $p_2'$, put $a_i' = \epsilon$. 
	
	Using induction on $m$, we show that two vertices are adjacent in $M_c$ if and only if they alternate in $p_1'p_2'p_3'$. 
	The statement holds for $m = 1$ by Lemma \ref{path_2k-1}. Assume that the statement holds if $M_c$ has at most $m - 1$ constituent paths. As shown in the proof of Theorem \ref{odd_parity}, it is sufficient to show that, for $a \in E_1 \cup E_2 \cup \ldots \cup E_{m-1}$ and $b \in E_{m} \setminus \{0,0'\}$, $a$ and $b$ are adjacent in $M_c$ if and only if $a$ and $b$ alternate in $p_1'p_2'p_3'$.
	
	Suppose  $a$ and $b$ are adjacent in $M_c$. Then, clearly $a \in \{0',0\}$ and $b \in \{a_{m}', a_{m}\}$. It can be observed that $ba\ll p_1'$, $ba \ll p_2'$, and $ba \ll  p_3'$ so that $a$ and $b$ alternate in $p_1'p_2'p_3'$.		
	
	Conversely, suppose that $a$ and $b$ are not adjacent in $M_c$. Then, $a$ must be an intermediate vertex of some path $E_i$ ($1 \le i \le m-1$). We show that  $a$ and $b$ do not alternate in $p_1'p_2'p_3'$ in the following three cases:
	\begin{itemize}
		\item Case 1: $b = a_{m}'$. If $a =  a_{i}'$, then $ab \ll p_1'$ and $ba \ll p_3'$; otherwise (i.e., $a \ne a_{i}'$), $ba \ll p_1'$ and $ab \ll p_2'$. 
		\item Case 2: $b = a_{m} $. If $a =  a_{i}$, then $ab \ll p_2'$ and $ba \ll p_3'$; otherwise, $ab \ll p_1'$ and $ba \ll p_2'$. 
		\item Case 3: $ b \not \in \{a_{m}, a_{m}'\}$. It can be observed that $ab \ll p_1'$ and $ba \ll p_3'$. 
	\end{itemize}
	Hence, $p_1'p_2'p_3'$ represents $M_c$ permutationally so that $\mathcal{R}^p(M_c) \le 3$.	\qed
\end{proof}

\begin{theorem}\label{00'adj_fc}
	Let a melon graph $M_c$ be a comparability graph. If the endpoints $0$ and $0'$ are adjacent in $M_c$, then $\mathcal{R}^p(M_c) \le 3$.
\end{theorem}

\begin{proof}
	Without loss of generality, let $M_c = (E_1,E_2, \ldots, E_m)$ such that $|E_1| \ge |E_2|\ge \cdots \ge |E_m|$, where $|E_i|$ is the number of vertices of $E_i$.  Since $E_m$ is of length one, note that the endpoints $0$ and $0'$ are adjacent in $M_c$. For $i \ge 2$, if a constituent path $E_i$ is of odd length, then the induced subgraph $M_c[E_i]$ is an even cycle. If $E_i$ is of even length, then it should be of length two (by Theorem \ref{chr_com_file}) so that $M_c[E_i]$ is a triangle. In the construction of permutations on the vertices of $M_c$, we require the intermediate vertices, say $x_i$ and $y_i$, adjacent to $a_i$ and $a_i'$ in $E_i$, respectively. 
	
	\noindent\textbf{Part-I.} Suppose all the constituent paths of $M_c$ are of odd parity.
	For $1 \le i < m$, note that $E_i$ is of length at least three. Consider the permutations $w_1 = 0'y_ia_i' u_i a_i x_i 0$, $w_2 = a_i v_i 0' a_i' 0$ and $w_3 = 0' a_i 0 v_i a_i'$, where $u_i$ and $v_i$ are permutations on the remaining intermediate vertices of $E_i$, as stated in Lemma \ref{even_cycle_word}. 
	Construct the following three permutations on the vertices of $M_c$: 
	\begin{align*}
		p_1' &= 0' y_{m-1} a_{m-1}' \cdots y_i a_i' \cdots y_1 a_1' u_{1} \cdots u_i  \cdots  u_{m-1}    a_1 x_1 \cdots a_i x_i   \cdots  a_{m-1} x_{m-1} 0 \\
		p_2' &= a_{m-1}v_{m-1}        \cdots  a_iv_i \cdots a_1 v_1    0'  a_{m-1}'    \cdots a_{i}' \cdots a_1'0\\
		p_3' &= 0'a_{m-1} \cdots a_{i} \cdots a_1 0 v_{1}a_{1}' \cdots v_ia_i' \cdots      v_{m-1} a_{m-1}',
	\end{align*} (for clarity, $1 \le i \le m-1$).
	If $E_i$ if of length three, note that $x_i = a_i'$ and $y_i = a_i$. Then, put $v_i = \epsilon$ in the permutations $p_2'$ and $p_3'$, and put $a_i' = y_i = \epsilon$ in $p_1'$. 
	
	Using induction on $m$, we show that two vertices are adjacent in $M_c$ if and only if they alternate in $p_1'p_2'p_3'$. 
	The statement holds for $m = 1$ by Lemma \ref{even_cycle_word}. Assume that the statement holds if $M_c$ has at most $m - 1$ constituent paths. It is sufficient to show that, for $a \in  E_2 \cup \ldots \cup E_{m-1} \cup E_m$ and $b \in E_{1} \setminus \{0,0'\}$, $a$ and $b$ are adjacent in $M_c$ if and only if $a$ and $b$ alternate in $p_1'p_2'p_3'$.
	
	Suppose  $a$ and $b$ are adjacent in $M_c$. Then, clearly $a \in \{0',0\}$ and $b \in \{a_{1}', a_{1}\}$. If $a = 0$, then $b = a_{1}$; in this case, we have $ba \ll p_1'$, $ba \ll p_2'$, and $ba \ll  p_3'$. In case $a = 0'$, then $b = a_{1}'$ and $ab \ll p_1'$, $ab \ll p_2'$, and $ab \ll p_3'$. Hence, in any case, $a$ and $b$ alternate in $p_1'p_2'p_3'$.

	Conversely, suppose that $a$ and $b$ are not adjacent in $M_c$. Then, $a$ must be an intermediate vertex of some path $E_i$ ($2 \le i \le m-1$). We show that  $a$ and $b$ do not alternate in $p_1'p_2'p_3'$ in the following three cases:
	
	\begin{itemize}
		\item $b \in \{a_1', y_1\}$. If $a \in \{a_i', y_i\}$, then $ab \ll p_1'$ and $ba \ll p_3'$. If $a \not  \in \{a_i', y_i\}$, then $ba \ll p_1'$ and $ab \ll p_2'$.
		
		\item $b \in \{a_1,x_1\}$. If $a  \in \{a_i, x_i\}$, then $ba \ll p_1'$ and $ab \ll p_2'$. Else if $a  \not \in \{a_i, x_i\}$, $ab \ll p_1'$ and $ba \ll p_3'$.
		
		\item $b \not \in \{a_1,a_1',x_1,y_1\}$, i.e., $b \ll u_1$ and $b \ll v_1$. If $a  \in \{a_i', y_i\}$, then $ab \ll p_1'$ and $ba \ll p_3'$. Else if $a \not \in \{a_i',y_i'\}$, then $ba \ll p_1'$and $ab \ll p_2'$.
	\end{itemize}
	Hence, if all the constituent paths of $M_c$ are of odd parity,  then $M_c$ can be represented by the word $p_1'p_2'p_3'$. 
	
	\noindent\textbf{Part-II.} Suppose that the melon graph $M_c$ has constituent paths of even length. Let $t$ be the least index such that $E_t$ is of length two. Then, the constituent paths  $E_t, E_{t+1} \ldots, E_{m-1}$ are of length two, with intermediate vertices, say $a_t, a_{t+1}, \ldots, a_{m-1}$, respectively. Construct the following three permutations on the vertices of $M_c$:
	\begin{align*}
		q_1' &= 0' y_{t-1} a_{t-1}' \cdots y_j a_j' \cdots y_1 a_1'  u_{1} \cdots u_j  \cdots  u_{t-1}    a_1 x_1 \cdots a_j x_j   \cdots  a_{t-1} x_{t-1}  a_{m-1} \cdots a_{i} \cdots a_{t} 0 \\
		q_2' &=  a_{t-1}v_{t-1}        \cdots  a_jv_j \cdots a_1 v_1    0'a_{t-1}'     \cdots a_{j}' \cdots a_1' a_t \cdots a_{i} \cdots a_{m-1} 0\\
		q_3' &= 0'  a_t \cdots a_{i} \cdots a_{m-1} a_{t-1} \cdots a_{j} \cdots a_1 0 v_{1}a_{1}' \cdots v_ja_j' \cdots      v_{t-1} a_{t-1}'.
	\end{align*} (for clarity, $t \le i \le m-1$ and $1 \le j \le t-1$). 
	For $F = E_1 \cup E_2 \cup \cdots \cup E_{t-1} \cup E_{m}$, observe that $(q_1'q_2'q_3')|_F = (p_1'p_2'p_3')|_F$. Thus, by Part-I, $(q_1'q_2'q_3')|_F$  represents $M_c[F]$ permutationally. It remains to show that, for $a = a_i$ ($t \le i\le m-1$) and $b \in E_j$ ($1 \le j \le m$), $a$ and $b$ are adjacent if and only if $a$ and $b$ alternate in $q_1'q_2'q_3'$.
	
	Suppose $a$ and $b$ are adjacent in $M_c$. Then, $b \in \{0,0'\}$. If $b = 0$, then clearly $ababab \ll q_1'q_2'q_3'$. Otherwise, $bababa \ll q_1'q_2'q_3'$. Thus, $a$ and $b$ alternate in $q_1'q_2'q_3'$. Conversely, suppose $a$ and $b$ are non-adjacent in $M_c$. If $a = a_i$ and $b = a_j$ (for $t \le i < j \le m-1$), then $a_j a_i \ll q_1'$ and $a_ia_j \ll q_2'$. If $a = a_i$ (for $t \le i \le m-1$) and $b \in E_1 \cup E_2 \cdots \cup  E_{j-1}$, then $ba\ll q_2'$ and $ab \ll q_3'$. Thus, $a$ and $b$ do not alternate in $q_1'q_2'q_3'$.  
	Hence, if $M_c$ has constituent paths of even length, then $M_c$ can be represented by the word $q_1'q_2'q_3'$.
	
	Hence, if the endpoints $0$ and $0'$ are adjacent in $M_c$, then $\mathcal{R}^p(M_c) \le 3$.	\qed
\end{proof}

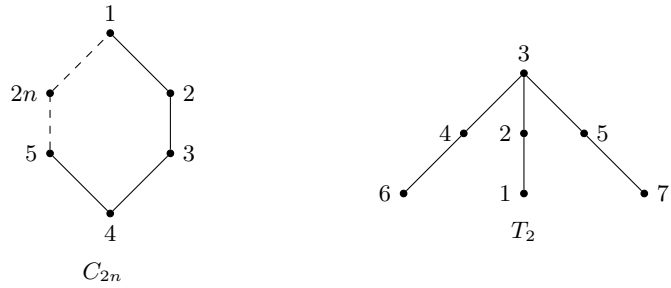
\begin{figure}[t]
	\centering	
	\begin{minipage}{.5\textwidth}
		\centering
		\begin{tikzpicture}[scale=0.8]
			\vertex (0) at (0,2) [label=above:$ 1$] {};  
			\vertex (1) at (1,1) [label=right:$ 2$] {}; 
			\vertex (2) at (1,0) [label=right:$3 $] {}; 
			\vertex (3) at (0,-1) [label=below:$4 $] {}; 
			\vertex (4) at (-1,0) [label=left:$5 $] {}; 
			\vertex (5) at (-1,1) [label=left:$2n $] {}; 		
			\path
			(0) edge (1)
			(1) edge (2)
			(2) edge (3)
			(3) edge (4);		
			\path[dashed]
			(4) edge (5)
			(5) edge (0);		
		\end{tikzpicture}
		
		$C_{2n}$
	\end{minipage}
	\begin{minipage}{.4\textwidth}
		\centering
		\begin{tikzpicture}[scale=0.8]
			\vertex (0) at (0,-3) [label=left:$ 1$] {};  
			\vertex (1) at (0,-2) [label=left:$ 2$] {}; 
			\vertex (2) at (0,-1) [label=above:$3 $] {}; 
			\vertex (3) at (-1,-2) [label=left:$4 $] {}; 
			\vertex (4) at (1,-2) [label=right:$5 $] {}; 
			\vertex (5) at (-2,-3) [label=left:$6$] {}; 
			\vertex (6) at (2,-3) [label=right:$ 7$] {};  
			\path
			(0) edge (1)
			(1) edge (2)
			(3) edge (2)
			(4) edge (2)
			(5) edge (3)
			(6) edge (4);	
		\end{tikzpicture}
		
		$T_2$
	\end{minipage}
	\caption{Forbidden induced subgraphs of melon graphs with \textit{prn} at most two}
	\label{t2_c2n}
\end{figure}

We now conclude this section with the classification of the \textit{prn} of $M_c$.

\begin{theorem}
	$\mathcal{R}^p(M_c) = 3$ if and only if it contains an even cycle $C_{2n}$, for $n \ge 3$, or $T_2$ (given in Fig. \ref{t2_c2n}) as an induced subgraph.
\end{theorem}
\begin{proof}
	In \cite{Gallai}, the permutation graphs, i.e., the class of graphs with \textit{prn} at most two, are characterized in terms of forbidden induced subgraphs (see, e.g., \cite[Theorem 7, Page 200]{Vincent}). Among the forbidden subgraphs for permutation graphs, only $C_{2n}$ (for $n \ge 3$) and $T_2$ are the possible induced subgraphs of $M_c$. Note that $T_2$ is an induced subgraph of $M_c$ if $M_c$ is a book graph with at least three pages (in which $C_4$ is a largest induced cycle). Accordingly, in view of Theorems \ref{odd_parity}, \ref{even_parity} and \ref{00'adj_fc}, since $\mathcal{R}^p(M_c) \le 3$, the result follows. 	\qed
\end{proof}    

\begin{remark}
	Among melon graphs, while $\mathcal{R}^p(K_n) = 1$ (for $n \le 3$), we have $\mathcal{R}^p(M_c) = 2$ if $M_c$ is the book graph with two pages or the constituent paths of $M_c$ are of length at most two or $M_c$ is the combination of both. 
\end{remark}    

\section{Line Graphs}
\label{line_melon}

\begin{figure}[t]
	\centering
	\begin{tabular}{ccc}
		\begin{tikzpicture}[scale=0.9]
			\vertex (0) at (0,0) [label=above:$0'$] {};  
			\vertex (1) at (-1,-1) [label=left:$1' $] {}; 
			\vertex (2) at (1,-1) [label=right:$2'  $] {}; 
			\vertex (3) at (-1,-2) [label=left:$1 $] {}; 
			\vertex (4) at (1,-2) [label=right:$2$] {};
			\vertex (5) at (-1, -1.5) [label=right:$ $] {};
			\vertex (6) at (1, -1.5) [label=right:$ $] {};
			
			\vertex (10) at (0,-3) [label=below:$ 0 $] {}; 
			\path
			(0) edge node[left] {$a_1$} (1)
			(0) edge node[right] {$a_2$}(2)
			(0) edge node[left] {$x$}(10)
			(10) edge node[left] {$a_{i+1}$} (3)
			(10) edge node[right] {$a_{i}$}(4)
			(5) edge node[left] {$ $}(1)
			(6) edge node[right] {$ $}(4);		
			\path[dashed]
			(1) edge (3)
			(2) edge (4);
		\end{tikzpicture}&\hspace{50pt} & \begin{tikzpicture}[scale=1.6]
			\vertex (a) at (0,0) [label=above:$a_1 $] {};  
			\vertex (b) at (1,0) [label=above:$a_2 $] {}; 
			\vertex (c) at (0,-1) [label=below:$a_{i+1} $] {}; 
			\vertex (d) at (1,-1) [label=below:$a_{i} $] {};
			\vertex (e) at (0,-0.5) [label=left:$  $] {}; 
			\vertex (f) at (1,-0.5) [label=right:$  $] {};		
			\vertex (0) at (0.5,-0.5) [label=right:$ x $] {}; 
			\path
			(a) edge (b)
			(c) edge (d)
			(0) edge (a)
			(0) edge (b)
			(0) edge (c)
			(0) edge (d)
			(a) edge (e)
			(d) edge (f);
			\path[dashed]
			(a) edge (c)
			(b) edge (d);
		\end{tikzpicture}\\
		$M$ & & $L(M)$
	\end{tabular}	
	\caption{A melon graph with three constituent paths and its line graph}
	\label{example_f}
\end{figure}
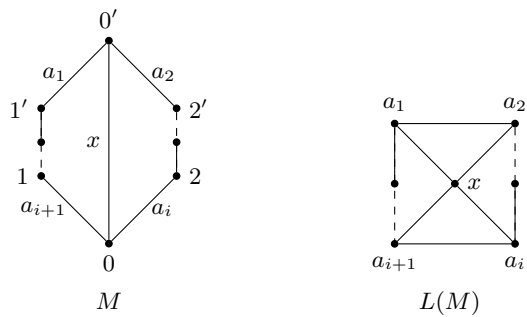

In this section, we characterize the word-representable line graphs of melon graphs. Consequently, we show that their representation number is at most three. Further, we also study the \textit{prn} concerning to these graphs.

The line graph of a graph $G = (V, E)$, denoted by $L(G)$, is the graph whose vertex set is $E$, and any two vertices are adjacent if and only if the corresponding edges of $G$ share a vertex.

If a melon graph $M$ has only two constituent paths, i.e., $M$ is a cycle, then the line graph $L(M)$ is isomorphic to $M$, and hence it is a 2-word-representable graph. In addition, if $M$ has three constituent paths where the endpoints $0$ and $0'$ are adjacent in $M$ (as shown in Fig. \ref{example_f}), we construct a 2-uniform word representing $L(M)$ as per the following:
\begin{enumerate}	
	\item Using the method given in \cite[Section 3.2]{kitaev15mono}, we construct a 2-uniform word $c_1c_nw$ for the cycle $\langle c_1, c_2, \ldots, c_i, c_{i+1}, \ldots,  c_n, c_1 \rangle$, where
	\[w = c_2 c_1 \cdots c_i c_{i-1} \cdots c_{n} c_{n-1} \] (for clarity, $2 \le i \le n$).	
	\item By inserting $x$ in between $c_2$ and $c_1$, and in between $c_{i+1}$ and $c_i$, consider the 2-uniform word 
	\[ c_1 c_n c_2 x c_1  \cdots c_i c_{i-1} c_{i+1} x  c_i \cdots  c_{n} c_{n-1} \]	
	and note that it represents $L(M)$. Hence, $L(M)$ is a 2-word-representable graph.
\end{enumerate}

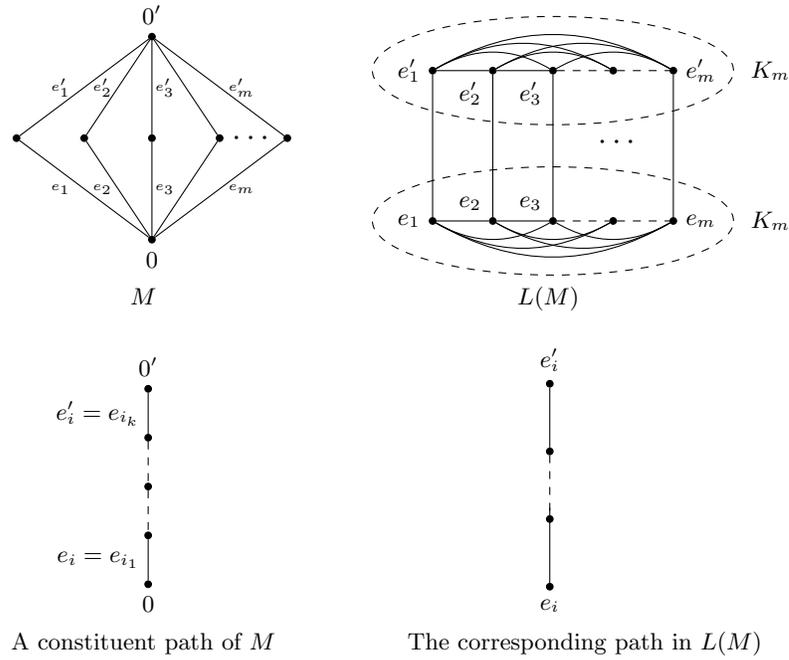
\begin{figure}[t]	
	\centering	
	\begin{tabular}{cc}
		\quad \qquad\begin{tikzpicture}[scale=0.9]
			\vertex (0) at (0,0.5) [label=above:$0'$] {};  
			\vertex (1) at (-2,-1) [label=left:$ $] {}; 
			\vertex (2) at (-1,-1) [label=left:$ $] {}; 
			\vertex (3) at (0,-1) [label=right:$ $] {}; 
			\vertex (4) at (2,-1) [label=left:$  $] {}; 
			\vertex (5) at (1,-1) [label=right:\large{$\cdots $}] {}; 		
			\vertex (10) at (0,-2.5) [label=below:$0$] {};  		
			\path
			(0) edge node[left=2] {\tiny{$e_1'$}} (1)
			(0) edge node[left=-1] {\tiny{$e_2'$}}(2)
			(0) edge node[right=-2] {\tiny{$e_3'$}} (3)
			(0) edge node[right] {\tiny{$e_m'$}}(4) 
			(10) edge node[left=2] {\tiny{$e_1$}}(1)
			(10) edge node[left=-1] {\tiny{$e_2$}} (2)
			(10) edge node[right=-2] {\tiny{$e_3$}}(3)
			(10) edge node[right] {\tiny{$e_m$}} (4)
			(0) edge  (5)
			(10) edge (5);		
		\end{tikzpicture}&   \quad \qquad\begin{tikzpicture}[scale=0.8]
			\vertex (1) at (-2,1) [label=left:$e_1'$] {}; 
			\vertex (2) at (-1,1) [label=below left:$e_2'$] {}; 
			\vertex (3) at (0,1) [label=below left :$e_3'$] {}; 
			\vertex (4) at (1,1) [label=above:$ $] {}; 
			\vertex (5) at (2,1) [label=right:$e_m'$] {}; 
			\vertex (11) at (-2,-1.5) [label=left:$e_1$] {}; 
			\vertex (12) at (-1,-1.5) [label=above left:$e_2$] {}; 
			\vertex (13) at (0,-1.5) [label=above left:$e_3$] {}; 
			\vertex (14) at (1,-1.5) [label=below:$ $] {}; 
			\vertex (15) at (2,-1.5) [label=right:$e_m$] {};
			\node (A) at (3,1) [label=right:$K_m $] {}; 
			\node (B) at (3,-1.5) [label=right:$K_m $] {}; 
			\node (0) at (1.1,-0.6) [label=above:\large{$\cdots $}] {}; 
			\draw (0,1)[dashed] ellipse  (3cm and 0.9cm) ;
			\draw (0,-1.5)[dashed] ellipse (3cm and 0.9cm);		
			\path	
			(1) edge (2)
			(1) edge[bend left] (3)
			(1) edge[bend left] (4)
			(1) edge[bend left] (5)
			(2) edge (3)
			(2) edge[bend left] (4)
			(2) edge[bend left] (5)
			(3) edge[dashed] (4)
			(3) edge[bend left] (5)
			(4) edge[dashed] (5)
			(11) edge (12)
			(11) edge[bend right] (13)
			(11) edge[bend right] (14)
			(11) edge[bend right] (15)
			(12) edge (13)
			(12) edge[bend right] (14)
			(12) edge[bend right] (15)
			(13) edge[dashed] (14)
			(13) edge[bend right] (15)
			(14) edge[dashed] (15)
			(11) edge (1)
			(12) edge (2)
			(13) edge (3)
			(15) edge (5); 
		\end{tikzpicture} \\
		\qquad\quad $M$ & $L(M)$ \\[10pt]
		\begin{tikzpicture}[scale=1.3]
			\vertex (0) at (0,0.5) [label=above:$0'$] {};  
			\vertex (1) at (0,0) [label=left:$ $] {}; 
			\vertex (2) at (0,-0.5) [label=left:$ $] {};
			\vertex (3) at (0,-1 ) [label=left:$ $] {};
			\vertex (10) at (0,-1.5) [label=below:$0$] {};  		
			\path
			(0) edge node[left] {$e_i' = e_{i_k}$} (1)
			(10) edge node[left] {$e_i = e_{i_1}$}(3);
			\path[dashed]
			(1) edge (2)
			(2) edge (3);		
		\end{tikzpicture}&   \begin{tikzpicture}[scale=0.6]
			\vertex (0) at (0,0.5) [label=above:$e_i'$] {};  
			\vertex (1) at (0,-1) [label=left:$ $] {}; 
			\vertex (2) at (0,-2.5) [label=left:$ $] {};
			\vertex (10) at (0,-4) [label=below:$e_i$] {};  		
			\path
			(0) edge  (1)
			(10) edge (3);		
			\path[dashed]
			(1) edge (2);		
		\end{tikzpicture} \\
		\quad \qquad A constituent path of $M$ &   \quad \qquad The corresponding path in $L(M)$
	\end{tabular}	
	\caption{A melon graph and its line graph}
	\label{line_cartesian}
\end{figure}

Suppose a melon graph $M =(E_1, \ldots, E_m)$ has all its constituent paths of length exactly two. Then, the line graph of the melon graph is the Cartesian product $K_m \square K_2$ of the complete graphs $K_m$ and $K_2$, as shown in Fig. \ref{line_cartesian}. In \cite{broere_2019}, it was shown that the representation number of $K_m \square K_2$ is at most three.

In the following, we show that the line graphs of the melon graphs whose constituent paths are of length at least two are 3-word-representable.

\begin{lemma}\label{non-adjacent_line}
	Let $M$ be a melon graph such that the endpoints $0$ and $0'$ are not adjacent. Then $L(M)$ is $3$-word-representable.
\end{lemma}
\begin{proof}
	If the constituent paths of $M$ are of length two, then $L(M) = K_m \square K_2$ as shown in Fig. \ref{line_cartesian}, and its representation number is at most three. In fact, the following 3-uniform word $w$ representing $K_m \square K_2$ can be constructed using the method given in \cite{broere_2019}:  
	\[w = e_1 \cdots e_i \cdots e_m e'_1 e_1 \cdots e'_i e_i \cdots e'_m e_m e'_1 \cdots e'_i \cdots e'_m e_1 e'_1 \cdots e_i e'_i \cdots e_m e'_m,\] (for clarity, $1 \le i \le m$).
	Note that, as shown in Fig. \ref{line_cartesian}, $e_i$ and $e_i'$ are adjacent in the line graph $L(M)$ corresponding to the two-length path $\langle e_i, e_i'\rangle$.
	
	Suppose $M$ has a constituent path $E_i$ of length at least three. For $k \ge 3$, let $E_i = \langle e_i = e_{i_1}, \ldots, e_{i_k} = e_i'\rangle$, where $e_{i_j}$, $1 \le j \le k$, are the edges. Note that the line graph $L(M)$ has a path of length $k-1$ corresponding to $E_i$ of $M$, which can be obtained by applying the subdivision (i.e.,  replacing an edge by a path) on the edge connecting $e_i$ and  $e_i'$ as shown in Fig. \ref{line_cartesian}, and extend the word $w$ representing $K_m \square K_2$.
	
	If the edge connecting $e_i$ and  $e_i'$ is replaced by a path of length two, say $\langle e_i, x, e'_i \rangle$, then replace the factors $e'_i e_i$ and $e_i e'_i $ in $w$ by $xe_ie'_ix$ and $e_i x e'_i$, respectively. Note that $e_i$ and $e'_i$ do not alternate in the word after the replacement. Between two consecutive occurrences of $x$ only $e_i$ and $e_i'$ occur exactly once. Hence, $x$ alternate only with $e_i$ and $e_i'$.
	
	If we replace the edge connecting $e_i$ and  $e_i'$ by a path of length $k$, for $k \ge 3$, then first replace the edge by a path of length two (with an intermediate vertex $x$) and note that the resultant graph is 3-word-representable, as described above.  Then delete the intermediate vertex $x$ and add the path of length $k$. By Theorem \ref{subdiv}, it can be observed that the resultant graph is 3-word-representable. \qed
\end{proof}

\begin{figure}
	\centering
	\begin{tabular}{ccc}
		\begin{tikzpicture}[scale=0.9]
			\vertex (0) at (0,0.5) [label=above:$0$] {};
			\vertex (1) at (-2,-1) [label=left:$1$] {};
			\vertex (2) at (-1,-1) [label=left:$2$] {};
			\vertex (3) at (1,-1) [label=right:$3$] {};
			\vertex (10) at (0,-2.5) [label=below:$0'$] {};
			\path
			(0) edge node[left] {\tiny{$e_1$}} (1)
			(0) edge node[right] {\tiny{$e_2$}} (2)
			(0) edge node[right] {\tiny{$e_3$}} (3)
			(0) edge node[right] {\tiny{$e_0$}} (10)
			(1) edge node[left] {\tiny{$e_1'$}} (10)
			(2) edge node[right] {\tiny{$e_2'$}} (10)
			(3) edge node[right] {\tiny{$e_3'$}} (10);
		\end{tikzpicture} & \hspace{40pt} & \begin{tikzpicture}[scale=1]
			\vertex (a) at (0,0) [label=left:${e_1}$] {};
			\vertex (b) at (0,-2) [label=left:$e_2$] {};
			\vertex (c) at (0.5,-1) [label=left:$e_3$] {};
			\vertex (d) at (1,-1) [label=below:$e_0$] {};
			\vertex (e) at (2,0) [label=right:$e_1'$] {};
			\vertex (f) at (2,-2) [label=right:$e_2'$] {};
			\vertex (g) at (1.5,-1) [label=right:$e_3'$] {};		
			\path
			(d) edge (a)
			(d) edge (b)
			(d) edge (c)
			(d) edge (d)
			(d) edge (e)
			(d) edge (f)
			(d) edge (g)
			(a) edge (b)
			(a) edge (c)
			(a) edge (e)
			(b) edge (c)
			(b) edge (f)
			(c) edge[bend left] (g)
			(e) edge (f)
			(e) edge (g)
			(f) edge (g);
		\end{tikzpicture}\\
		\quad  \quad$A_3$ && $L(A_3)$
	\end{tabular}	
	\caption{A triangular book graph and its line graph}
	\label{line_g}
\end{figure}
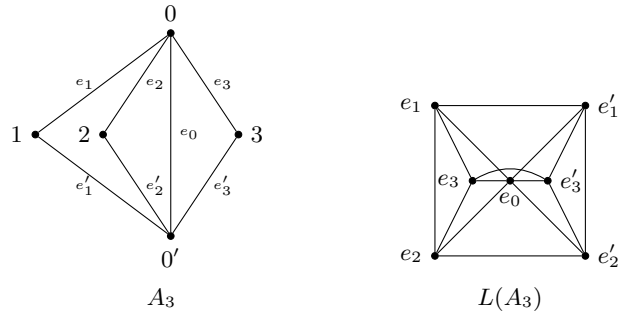

\begin{lemma}\label{l_a3}
	Let $A_3$ be the triangular book graph with three pages, as shown in Fig. \ref{line_g}. The line graph $L(A_3)$ is not word-representable.
\end{lemma}
\begin{proof}
	Suppose $L(A_3)$ is word-representable. Then, by \cite[Theorem 10]{kitaev08}, the subgraph induced by the neighborhood of every vertex of $L(A_3)$ must be a comparability graph. However, it can be observed that the subgraph of $L(A_3)$ induced by the neighborhood of the vertex $e_0$ is the prism $Pr_3$  (cf. Fig. \ref{line_g}), which is not a comparability graph. For instance, the complement of $Pr_3$ is the graph $C_6$. Since $C_6$ is a comparability graph with \textit{prn} three, $Pr_3$ is not a comparability graph, by \cite[Theorem 2.1]{trotterbook} (from \cite{Gallai}). Thus, $L(A_3)$ is not word-representable. \qed
\end{proof}

\begin{lemma}\label{H}
	The graph $H = (V, E)$, where $V  = \{a_1,a_2, \ldots, a_m, b_1,b_2, \ldots , b_m, x\}$ and $E = \Big\{\{a_i,a_j\}\; |\; i \ne j\Big\} \cup \Big\{\{b_i, b_j\}\; |\; i \ne j\Big\} \cup \Big\{\{x, y\}\; |\; y \in V \setminus \{x\} \Big\} \cup \Big\{\{a_1, b_1\}, \{a_2, b_2\}\Big\}$, is permutationally 3-representable.  
\end{lemma}

\begin{proof}
	Consider the following three permutations over the set $V$:
	\begin{align*}
		q_1 &= a_1 b_2 a_3 b_3 \cdots a_j b_j \cdots a_m b_m a_2 b_1x\\
		q_2 &= a_1 a_3 \cdots a_j \cdots a_m b_2 a_2 b_3 \cdots b_j \cdots b_m b_1x\\
		q_3 &= b_2 b_3 \cdots b_j \cdots b_m a_1 b_1 a_3 \cdots a_j \cdots a_m a_2x, 
	\end{align*} (for clarity, $3 \le j \le m$).	
	We show that the word $w = q_1q_2q_3$ represents $H$ permutationally. 
	
	Note that $x$ occurs at the end of each permutation in $w$ so that $x$ alternates with all $a_i$'s and $b_i$'s in $w$. Since $a_1 a_3\cdots a_m a_2 \ll q_i$, for all $i = 1, 2, 3$, we have $a_r$ and $a_s$, for $r \ne s$, alternate in $w$. Similarly, for each for $i = 1, 2, 3$, we have $b_2 b_3 b_4 \cdots b_m b_1\ll q_i$ so that $b_r$ and $b_s$, for $r \ne s$, alternate in $w$. Further note that $a_1b_1a_1b_1a_1b_1 \ll w$ and $b_2a_2b_2a_2b_2a_2 \ll w$. Hence, any two adjacent vertices of $H$ alternate in $w$. 
	
	We consider all the non-adjacent pairs of vertices in $H$ and observe that they do not alternate in $w$.  For all $r, s \in \{3, \ldots, m\}$, $a_r$ and $b_s$ are not adjacent in $H$ and $a_rb_sb_s a_r \ll q_2 q_3 \ll w$.  
	Further, for $2 \le r \le m$, $a_1$ is not adjacent to $b_r$ and note that $a_1 b_r b_r a_1 \ll q_2 q_3 \ll w$. Also,  $a_2$ is not adjacent to $b_s$, for $s \in \{1,3,\ldots, m\}$ and,  clearly $a_2 b_s b_s a_2 \ll q_2 q_3 \ll w$.  Next,
	note that $b_1$ is not adjacent to $a_r$ for $2 \le r \le m$ and we have $a_rb_1b_1 a_r \ll q_2 q_3$. Finally, observe that $b_2$ is not adjacent to $a_s$, for $s \in \{1,3,\ldots,m\}$ and $a_sb_2b_2a_s \ll q_2 q_3$. \qed
\end{proof}

\begin{figure}[t]	
	\centering	
	\begin{tabular}{cc}
		\quad \qquad\begin{tikzpicture}[scale=1.2]
			\vertex (0) at (0,0) [label=above: {$0'$}] {};  
			\vertex (1) at (0.5,-0.3) [label=below:$ $] {};
			\vertex (2) at (0.7,-1) [label=left:{\tiny{$E_1$}}] {};
			\vertex (3) at (0.5,-1.7) [label=above:$ $] {};
			\vertex (10) at (0,-2.1) [label=below: {$0 $}] {};		
			\vertex (11) at (0.9,-0.1) [label=above:$ $] {};  
			\vertex (12) at (1.25,-0.5) [label=below:$ $] {};
			\vertex (13) at (1.35,-1) [label=left:\tiny{$E_2 $}] {};
			\vertex (14) at (1.25,-1.5) [label=above:$ $] {};
			\vertex (15) at (1,-1.9) [label=below:$ $] {};
			\vertex (21) at (1.3, 0.1) [label=above:$ $] {};  
			\vertex (22) at (1.7,-0.1) [label=below:$ $] {};
			\vertex (23) at (2.1,-0.5) [label=left:$ $] {};
			\vertex (24) at (2.25,-1) [label=right:\tiny{$E_m$}] {};
			\vertex (25) at (2.1,-1.5) [label=above:$ $] {};  
			\vertex (26) at (1.8,-1.9) [label=below:$ $] {};
			\vertex (27) at (1.3,-2.2) [label=left:$$] {};
			\node (A) at (1.4,-1) [label=right:$\cdots$] {};		
			\path 
			(0) edge [bend left=20] (1)
			(1) edge [bend left=10] (2)
			(2) edge [bend left=10] (3)
			(3) edge [bend left=10] (10)
			(0) edge [bend left=20] (11)
			(11) edge [bend left=10] (12)
			(12) edge [bend left=10](13)
			(13) edge [bend left=5] (14)
			(14) edge [bend left=10](15)
			(15) edge [bend left=20] (10)
			(0) edge [bend left=20]  (21)
			(21) edge [bend left=10] (22)
			(22) edge [bend left=10]  (23)
			(23) edge [bend left=10]  (24)
			(24) edge [bend left=10]  (25)
			(25) edge [bend left=10]  (26)
			(26) edge [bend left=10]  (27)
			(27) edge [bend left=20]  (10)
			(0) edge node[ left=-1] {\tiny{$E_0$}} (10);		
		\end{tikzpicture}&   \quad \qquad\begin{tikzpicture}[scale=0.8]
			\vertex (1) at (-2,1) [label= left:\tiny{$a_1$}] {}; 
			\vertex (2) at (-1,1) [label=left:\tiny{$a_2$}] {}; 
			\vertex (3) at (0,1) [label=below left : $ $] {}; 
			\vertex (4) at (1,1) [label=above:$ $] {}; 
			\vertex (5) at (2,1) [label= right:\tiny{$a_m$}] {}; 		
			\vertex (11) at (-2,-1.5) [label= left:\tiny{$b_1$}] {}; 
			\vertex (12) at (-1,-1.5) [label= left:\tiny{$b_2$}] {}; 
			\vertex (13) at (0,-1.5) [label=above left:$ $] {}; 
			\vertex (14) at (1,-1.5) [label=below:$ $] {}; 
			\vertex (15) at (2,-1.5) [label=right:\tiny{$b_m$}] {}; 		
			\vertex (0) at (-4,-0.3) [label=left:\tiny{$x$}] {};		
			\node (A) at (3,1) [label=right: {$K_m $}] {}; 
			\node (B) at (3,-1.5) [label=right: {$K_m $}] {}; 	
			\node (c) at (0,-0.2) [label=right:\Large{$\cdots$}] {}; 
			\vertex (21) at (-2.1, 0.2) [label= left:$ $ ] {}; 
			\vertex (22) at (-1.9,-0.9) [label=left:$ $] {};  		
			\vertex (31) at (-0.9, 0.5) [label= left:$ $ ] {}; 
			\vertex (32) at (-1,-0.1) [label=left:$ $] {}; 
			\vertex (33) at (-1.1, -0.6) [label= left:$ $ ] {}; 
			\vertex (34) at (-0.8,-1) [label=left:$ $] {}; 		
			\vertex (41) at (2.1, 0.7) [label= left:$ $ ] {}; 
			\vertex (42) at (2.2,0.3 ) [label=left:$ $] {}; 
			\vertex (43) at (2.1, -0.1) [label= left:$ $ ] {}; 
			\vertex (44) at (2,-0.5) [label=left:$ $] {}; 		
			\vertex (45) at (2.1, -0.8) [label= left:$ $ ] {}; 
			\vertex (46) at (2,-1.2) [label=left:$ $] {}; 
			\draw (0,1)[dashed] ellipse  (3cm and 0.9cm) ;
			\draw (0,-1.5)[dashed] ellipse (3cm and 0.9cm);		
			\path
			(1) edge[bend left] (2)
			(1) edge[bend left] (3)
			(1) edge[bend left] (4)
			(1) edge[bend left] (5)	
			(2) edge[bend left] (4)
			(2) edge[bend left] (5)
			(2) edge[bend left] (3)
			(3) edge[bend left] (4)
			(3) edge[bend left] (5)
			(4) edge[bend left] (5)
			(2) edge[bend left] (4)
			(11) edge[bend right] (12)
			(11) edge[bend right] (13)
			(11) edge[bend right] (14)
			(11) edge[bend right] (15)	
			(12) edge[bend right] (14)
			(12) edge[bend right] (15)
			(12) edge[bend right] (13)	
			(13) edge[bend right] (15)
			(13) edge[bend right] (14)
			(14) edge[bend right] (15)
			(12) edge[bend right] (14)		
			(1) edge[bend right =10] (21)
			(21) edge [bend left =10] node[left=-2] {\tiny{$E_1'$}} (22)
			(22) edge [bend left =5] (11)		
			(2) edge[bend left =12] (31)
			(31) edge [bend left =10] (32)
			(32) edge [bend right =15] node[right=-2] {\tiny{$E_2'$}}(33)
			(33) edge[bend left =10] (34)
			(34) edge [bend left =10] (12)		
			(5) edge[bend left =2] (41)
			(41) edge [bend left =2] (42)
			(42) edge [bend left =2] (43)
			(43) edge[bend right =10] node[right=-1] {\tiny{$E_m'$}} (44)
			(44) edge [bend left =10]  (45)
			(45) edge[bend left =10]  (46)		
			(46) edge [bend right =10] (15);		
			\path[dotted]
			(0) edge(1)
			(0) edge (2)		
			(0) edge (5)
			(0) edge(11)
			(0) edge (12)		
			(0) edge (15);
		\end{tikzpicture} \\[5pt]
		$M$ &  \qquad \qquad \quad$L(M)$
	\end{tabular}	
	\caption{A melon graph and its line graph}
	\label{L_M}	
\end{figure}
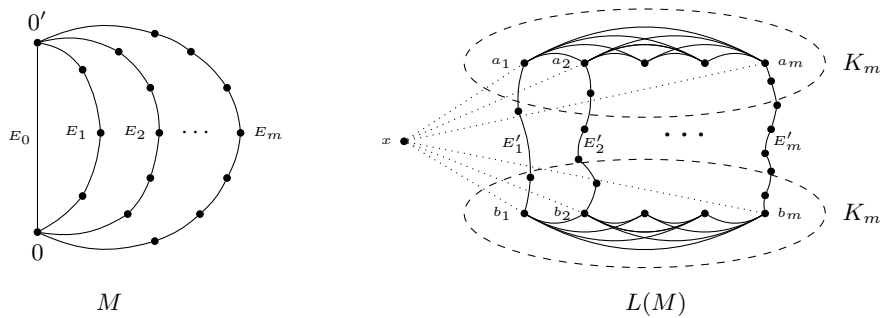

\begin{lemma}\label{adjacent_line}
	Let $M$ be a melon graph in which the endpoints $0$ and $0'$ are adjacent and that does not contain $A_3$ as an induced subgraph. Then, $L(M)$ is $3$-word-representable.
\end{lemma}
\begin{proof}
	Let $M$ be $(E_0, E_1, \ldots, E_m)$ such that the constituent paths are arranged in the non-decreasing order by their lengths. Thus, $E_0$ is the constituent path of length one. Further, since $M$ does not contain $A_3$ as an induced subgraph, $M$ must have at most two constituent paths each of length two, i.e., $E_1$ and $E_2$ are possibly of length two, and the remaining are of length at least three. 
	
	As shown in Fig. \ref{L_M}, note that $L(M)$ has two copies of $K_m$, with vertex sets, say $\{a_1, \ldots, a_m\}$ and $\{b_1, \ldots, b_m\}$, as induced subgraphs connected by paths  $E_i'$ ($1 \le i \le m$) between $a_i$ and $b_i$ ($1 \le i \le m$) and having a vertex  $x$ adjacent to all the vertices $a_i$'s and $b_i$'s. Here, the vertex $x$ corresponds to the path $E_0$ of $M$ and, for $1 \le i \le m$, the path $E_i'$ corresponds to the path $E_i$ of $M$ and of length one less than the length of $E_i$.         
	
	In the following, we construct a 3-uniform word representing $L(M)$ by iteratively extending the word $w$ representing the graph $H$ given in Lemma \ref{H}. 
	
	Case-1: Suppose the length of $E_i$ (for $i = 1, 2$) is two. Note that $H$ is clearly an induced subgraph of $L(M)$ in which the edges $\{a_1, b_1\}$ and $\{a_2, b_2\}$ are the paths $E_1'$ and $E_2'$, respectively. Set $H' = H$ and $w' = w$.
	
	Case-2: Suppose $E_1$ is of length three. Update the path $E_1'$ in the graph $H$ with an intermediate vertex $y_1$ so that $E_1' = \langle a_1, y_1, b_1\rangle$ is a path of length two. Let $H_1$ be the resultant graph. Accordingly, extend the word $w$ to $w_1$ by replacing the factor $a_1b_1$ with the word $y_1 b_1 a_1 y_1$, and the first occurrence of $b_1$ with $b_1y_1$. Note that the replacement of the word $y_1 b_1 a_1 y_1$ ensures that the vertices $a_1$ and $b_1$ do not alternate and $y_1$ alternates only with $a_1$ and $b_1$. Hence, the 3-uniform word $w_1$ represents $H_1$. 
	
	Case-3: Suppose $E_2$ is of length three. If $E_1$ is of length two, then set $w_1 = w$ and $H_1 = H$. Otherwise, consider the graph $H_1$ and its word-representant $w_1$ from Case-2.  Update the path $E_2'$ in the graph $H_1$ with an intermediate vertex $y_2$ so that $E_2' = \langle a_2, y_2, b_2\rangle$ is a path of length two. Let $H_2$ be the resultant graph. In the word $w_1$, replace the factor $b_2a_2$ with the word $y_2 a_2b_2y_2$ and the third occurrence of $a_2$ with $a_2y_2$. Note that the resultant 3-uniform word $w_2$ represents $H_2$. Set $H' = H_2$ and $w' = w_2$.
	
	Case-4: For $i = 1$ or $2$, suppose $E_i$ is of length $k$, where $k \ge 4$. Consider the graph $H_2$ from Case-3 and its word-representant $w_2$. Remove the intermediate vertex $y_i$ of $E_i'$ in $H_2$ and connect $a_i$ and $b_i$ with a path of length $k-1$. By Theorem \ref{subdiv}, the resultant graph, say $H'$, is 3-word-representable. Let $w'$ be the word-representant of $H'$. 
	
	Suppose each of $E_3, \ldots, E_t$ (for some $t \le m$) is of length three. Update the graph $H'$ by the paths $E_i' = \langle a_i,y_i, b_i \rangle$, for $3 \le i \le t$. Let $H''$ be the resultant graph. Accordingly, extend $w'$ to $w''$ by replacing the factor $a_ib_i$ in $w'$ with the word $y_i a_i b_i y_i$ and the second occurrence of $b_i$ by $b_i y_i$.  For $3 \le i \le t$, the factor $y_i a_i b_i y_i$ ensures that $y_i$ alternates only with $a_i$ and $b_i$ in $w''$. Hence, the 3-uniform word $w''$ represents $H''$.
	
	For the paths $E_{t+1}, \ldots, E_m$, which are of length at least four in $M$, consider the graph $H''$ and add paths $E_i'$ between the vertices $a_i$ and $b_i$. Note that the resultant graph is $L(M)$. Since $H''$ is 3-word-representable, by Theorem \ref{subdiv}, $L(M)$ is 3-word-representable.	\qed	 
\end{proof}

From lemmas \ref{non-adjacent_line}, \ref{l_a3} and \ref{adjacent_line}, we conclude the following characterization of the word-representability of the line graph of a melon graph. 

\begin{theorem}\label{main1}
	Let $M$ be a melon graph. The line graph $L(M)$ is word-representable if and only if $M$ does not contain $A_3$ as an induced subgraph.
\end{theorem}

Further, we have the following result for the representation number of the line graph of a melon graph.

\begin{theorem}\label{main2}
	For a melon graph $M$, if $L(M)$ is word-representable, then $\mathcal{R}(L(M)) \le 3$. Furthermore, $\mathcal{R}(L(M)) = 3$ if and only if $M$ consists of three constituent paths of length at least two. 
\end{theorem}

Note that the graph $K_3 \square K_2$ is not a circle graph. Further, any melon graph $M$ with three constituent paths of length at least two always contains $K_3 \square K_2$ as a vertex-minor. Hence, we have the following corollary of Theorem \ref{main2}.

\begin{corollary}
	For a melon graph $M$, let $L(M)$ be word-representable. Then, $\mathcal{R}(L(M)) = 3$ if and only if $L(M)$ contains $K_3 \square K_2$ as a vertex-minor.
\end{corollary}

\begin{figure}[t]
	\centering
	\begin{tabular}{cc}
		\begin{tikzpicture}[scale=1]
			\vertex (1) at (0,3) [label=above:$ $] {};  
			\vertex (2) at (-1,2) [label=left:$ $] {}; 
			\vertex (4) at (1,2) [label=right:$ $] {}; 
			\vertex (5) at (0,1) [label=below:$ $] {}; 
			\path
			(1) edge node[left] {$e_1$} (2)
			(1) edge node[left=-2] {$e_0$} (5)
			(1) edge node[right] {$e_2$} (4)
			(2) edge node[left] {$e_1'$} (5)
			(4) edge node[right] {$e_2'$} (5);
		\end{tikzpicture}& \qquad \qquad \qquad \begin{tikzpicture}[scale=0.8]
			\vertex (1) at (0,0) [label=below:$e_1'$] {};  
			\vertex (2) at (2,0) [label=below:$e_2'$] {}; 
			\vertex (3) at (1,1) [label=left:$e_0$] {}; 
			\vertex (4) at (0,2) [label=above:$e_1$] {}; 
			\vertex (5) at (2,2) [label=above:$e_2$] {}; 		
			\path
			(1) edge (2)
			(1) edge (4)
			(2) edge (5)
			(4) edge (5)
			(3) edge (1)
			(3) edge (2)
			(3) edge (4)
			(3) edge (5) ;
		\end{tikzpicture}\\
		$A_2$ & \qquad \qquad $L(A_2)$ \\
	\end{tabular}
	\caption{Triangular book graph with two pages and its corresponding line graph}
	\label{com_book}
\end{figure}
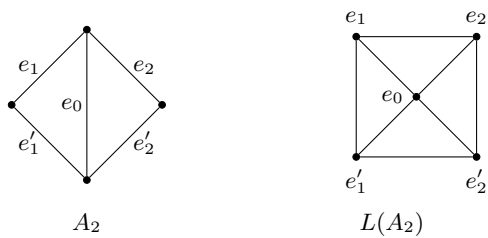

Finally, we characterize the comparability graphs within the class of line graphs of melon graphs. Also, we present the \textit{prn} of the respective graphs.

\begin{theorem}
	Among the line graphs of melon graphs, the following are the only comparability graphs: $L(P_n)$, $L(C_{2n})$, for $n \ge 2$, and $L(A_2)$. Here, $P_n$ is the path on $n$ vertices, $C_{2n}$ is the cycle on $2n$ vertices and $A_2$ is the graph given in Fig. \ref{com_book}.	 Moreover, the following statements regarding the \textit{prn} are evident.   
	\begin{enumerate}
		\item $\mathcal{R}^p(L(P_2)) = 1$ and $\mathcal{R}^p(L(P_n)) = 2$, for $n \ge 3$. 
		\item $\mathcal{R}^p(L(C_4)) = 2$ and $\mathcal{R}^p(L(C_{2n})) = 3$, for $n \ge 3$. 
		\item  $\mathcal{R}^p(L(A_2)) = 2$. 
	\end{enumerate}
\end{theorem}
\begin{proof}
	Suppose $M$ is a melon graph with three constituent paths of length at least two. Clearly, the line graph of $M$ contains $K_3 \square K_2$ as a vertex minor (see Fig. \ref{line_cartesian}). Note that $K_3 \square K_2$ is the prism $Pr_3$, which is not a comparability graph (see Lemma \ref{l_a3}). Also, note that the graph obtained by applying subdivision on the edges joining two copies of $K_3$ in  $K_3 \square K_2$ contains $S_1$ or $S_2$ given in Fig. \ref{non_com} as an induced subgraph. It is known that $S_1$ and  $S_2$ are not comparability graphs (see the characterization of comparability graphs in \cite{Gallai}). Thus, $M$ is not a comparability graph. Hence, $M$ can have at most two paths or, whenever $M$ has three paths, it has a path of length one. 
	
	If $M$ has only one constituent path, then $M = P_n$, for some $n \ge 2$. Then, $L(M) = P_{n-1}$, and hence, it is a permutation graph. 
	
	Suppose $M = (E_1, E_2)$. As $M$ is a cycle, $L(M)$ is isomorphic to $M$. Note that odd cycles are not comparability graphs. Hence, in this case, the line graph of $M$ is a comparability graph if both $E_1$ and $E_2$ are of the same parity, i.e., $M$ is an even cycle.
	
	Suppose $M = (E_1, E_2, E_3)$ is a melon graph, where $E_1$ is of length one. We have the following cases:
	\begin{itemize}
		\item $E_2$ and $E_3$ are not of the same parity: The line graph of $M$ contains an odd cycle of length at least five as an induced subgraph. Thus, $L(M)$ is not a comparability graph.		
		\item $E_2$ and $E_3$ are of odd parity: From Fig.  \ref{example_f}, it can be observed that the line graph $L(M)$ contains either $S_1$ or $S_2$, given in Fig. \ref{non_com}, as an induced subgraph. Thus, $L(M)$ is not a comparability graph.	
		\item $E_2$ and $E_3$ are of even parity: 
		\begin{itemize}
			\item[-] If $E_2$ or $E_3$ is of length at least four, then $M$ contains an odd cycle of length at least five, which includes $E_1$. Accordingly, the line graph of $M$ will also contain an odd cycle as an induced subgraph so that $L(M)$ is not a comparability graph.			
			\item[-] If both $E_2$ and $E_3$ are of length two then $M = A_2$, as shown in Fig. \ref{com_book}. Accordingly, $L(A_2)$ is a permutation graph (refer to the forbidden induced subgraphs of permutation graphs in \cite{Gallai}).
		\end{itemize}
	\end{itemize}
	\qed
\end{proof}

\begin{figure}[t]
	\centering
	\begin{tabular}{cc}
		\begin{tikzpicture}[scale=0.7]
			\vertex (1) at (0,0) [label=below:$ $] {};  
			\vertex (2) at (-1,-1.3) [label=below:$ $] {}; 
			\vertex (3) at (1,-1.3) [label=left:$ $] {}; 
			\vertex (4) at (-1.8,-2.3) [label=above:$ $] {}; 
			\vertex (5) at (1.8,-2.3) [label=above:$ $] {}; 
			\vertex (6) at (0,0.7) [label=above:$ $] {}; 		
			\path
			(1) edge (2)
			(1) edge (3)
			(2) edge (3)
			(2) edge (4)
			(3) edge (5)
			(1) edge (6);
		\end{tikzpicture}  & \qquad \qquad
		\begin{tikzpicture}[scale=0.7]
			\vertex (1) at (0,0) [label=below:$ $] {};  
			\vertex (2) at (-1,-1) [label=below:$ $] {}; 
			\vertex (3) at (1,-1) [label=left:$ $] {}; 
			\vertex (4) at (-1,-2.3) [label=above:$ $] {}; 
			\vertex (5) at (1,-2.3) [label=above:$ $] {}; 
			\vertex (6) at (0,0.7) [label=above:$ $] {}; 		
			\path
			(1) edge (2)
			(1) edge (3)
			(2) edge (3)
			(2) edge (4)
			(3) edge (5)
			(1) edge (6)
			(4) edge (5);
		\end{tikzpicture}\\
		$S_1$ & \qquad \qquad $S_2$
	\end{tabular}
	\caption{Non-comparability graphs}
	\label{non_com}
\end{figure}
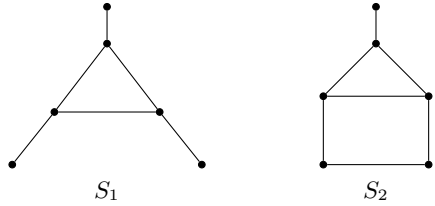

\section{Conclusion}

The melon graphs are series-parallel graphs. It is known that the class of series-parallel graphs is a subclass of 3-colorable planar graphs (see \cite{dissaux}), and hence a simple series-parallel graph is word-representable. Accordingly, one may try to extend the work presented in this paper to series-parallel graphs. More precisely, the following questions arise in this connection. 
\begin{enumerate}
	\item What is the representation number of a series-parallel graph?
	\item Being planar comparability graphs, the \textit{prn} of permutationally representable series-parallel graphs is at most four. Thus, one can ask for their classification between the \textit{prn} three and four. 
	\item Determine a finite list of forbidden (induced) subgraphs for characterizing the word-representability of line graphs of series-parallel graphs.
\end{enumerate}

In the literature, various authors have characterized   the word-representability of some subclasses of planar graphs through operations such as triangulation and face subdivison on a known word-representable planar graph \cite{Akrobotu_2015,chen_2016,Glen_2017}. In this work, we characterized the word-representability of line graph operation on  melon graphs. It will be interesting to study the word-representability of triangulation or face-subdivision of melon graphs.

\section{Acknowledgements}
	The authors are thankful to the referee for his/her comments which have improved the manuscript significantly. The first author is also thankful to the Council of Scientific and Industrial Research (CSIR), Government of India, for awarding the research fellowship for pursuing Ph.D. at IIT Guwahati.

%\bibliographystyle{abbrv}
%\bibliography{thesisref}

\begin{thebibliography}{10}
	
	\bibitem{Akgun_2019}
	O.~Akg\"{u}n, I.~Gent, S.~Kitaev, and H.~Zantema.
	\newblock Solving computational problems in the theory of word-representable
	graphs.
	\newblock {\em J. Integer Seq.}, 22(2):Art. 19.2.5, 18, 2019.
	
	\bibitem{Akrobotu_2015}
	P.~Akrobotu, S.~Kitaev, and Z.~Mas\'{a}rov\'{a}.
	\newblock On word-representability of polyomino triangulations.
	\newblock {\em Siberian Adv. Math.}, 25(1):1--10, 2015.
	
	\bibitem{Baker_1972}
	K.~A. Baker, P.~C. Fishburn, and F.~S. Roberts.
	\newblock Partial orders of dimension {$2$}.
	\newblock {\em Networks}, 2:11--28, 1972.
	
	\bibitem{biro_2021}
	C.~Bir\'o, B.~o. Bosek, H.~C. Smith, W.~T. Trotter, R.~Wang, and S.~J. Young.
	\newblock Planar posets that are accessible from below have dimension at most
	6.
	\newblock {\em Order}, 38(1):21--36, 2021.
	
	\bibitem{bouchet}
	A.~Bouchet.
	\newblock Circle graph obstructions.
	\newblock {\em J. Combin. Theory Ser. B}, 60(1):107--144, 1994.
	
	\bibitem{Brause_2017}
	C.~Brause, A.~Kemnitz, M.~Marangio, A.~Pruchnewski, and M.~Voigt.
	\newblock Sum choice number of generalized {$\theta$}-graphs.
	\newblock {\em Discrete Math.}, 340(11):2633--2640, 2017.
	
	\bibitem{broere_2019}
	B.~Broere and H.~Zantema.
	\newblock The {$k$}-dimensional cube is {$k$}-representable.
	\newblock {\em J. Autom. Lang. Comb.}, 24(1):3--12, 2019.
	
	\bibitem{Carraher_2015}
	J.~M. Carraher, T.~Mahoney, G.~J. Puleo, and D.~B. West.
	\newblock Sum-paintability of generalized theta-graphs.
	\newblock {\em Graphs Combin.}, 31(5):1325--1334, 2015.
	
	\bibitem{chen_2016}
	H.~Z.~Q. Chen, S.~Kitaev, and B.~Y. Sun.
	\newblock Word-representability of face subdivisions of triangular grid graphs.
	\newblock {\em Graphs Combin.}, 32(5):1749--1761, 2016.
	
	\bibitem{dissaux}
	T.~Dissaux, G.~Ducoffe, N.~Nisse, and S.~Nivelle.
	\newblock Treelength of series-parallel graphs.
	\newblock {\em Discrete Appl. Math.}, 341:16--30, 2023.
	
	\bibitem{Dennis_2000}
	D.~Eichhorn, D.~Mubayi, K.~O'Bryant, and D.~B. West.
	\newblock The edge-bandwidth of theta graphs.
	\newblock {\em J. Graph Theory}, 35(2):89--98, 2000.
	
	\bibitem{Felsner_2015}
	S.~Felsner, W.~T. Trotter, and V.~Wiechert.
	\newblock The dimension of posets with planar cover graphs.
	\newblock {\em Graphs Combin.}, 31(4):927--939, 2015.
	
	\bibitem{Gallai}
	T.~Gallai.
	\newblock Transitiv orientierbare {G}raphen.
	\newblock {\em Acta Math. Acad. Sci. Hungar.}, 18:25--66, 1967.
	
	\bibitem{Glen_2017}
	M.~Glen and S.~Kitaev.
	\newblock Word-representability of triangulations of rectangular polyomino with
	a single domino tile.
	\newblock {\em J. Combin. Math. Combin. Comput.}, 101:131--144, 2017.
	
	\bibitem{halldorsson11}
	M.~M. Halld\'{o}rsson, S.~Kitaev, and A.~Pyatkin.
	\newblock Alternation graphs.
	\newblock In {\em Graph-theoretic concepts in computer science}, volume 6986 of
	{\em Lecture Notes in Comput. Sci.}, pages 191--202. Springer, Heidelberg,
	2011.
	
	\bibitem{Joret_2017}
	G.~Joret, P.~Micek, and V.~Wiechert.
	\newblock Planar posets have dimension at most linear in their height.
	\newblock {\em SIAM J. Discrete Math.}, 31(4):2754--2790, 2017.
	
	\bibitem{kitaev13}
	S.~Kitaev.
	\newblock On graphs with representation number 3.
	\newblock {\em J. Autom. Lang. Comb.}, 18(2):97--112, 2013.
	
	\bibitem{kitaev15mono}
	S.~Kitaev and V.~Lozin.
	\newblock {\em Words and graphs}.
	\newblock Monographs in Theoretical Computer Science. An EATCS Series.
	Springer, Cham, 2015.
	
	\bibitem{kitaev08}
	S.~Kitaev and A.~Pyatkin.
	\newblock On representable graphs.
	\newblock {\em J. Autom. Lang. Comb.}, 13(1):45--54, 2008.
	
	\bibitem{kitaev_linegraphs}
	S.~Kitaev, P.~Salimov, C.~Severs, and H.~Ulfarsson.
	\newblock Word-representability of line graphs.
	\newblock {\em Open J. Discrete Math.}, 1(2):96--101, 2011.
	
	\bibitem{kitaev08order}
	S.~Kitaev and S.~Seif.
	\newblock Word problem of the {P}erkins semigroup via directed acyclic graphs.
	\newblock {\em Order}, 25(3):177--194, 2008.
	
	\bibitem{Vincent}
	V.~Limouzy.
	\newblock Seidel minor, permutation graphs and combinatorial properties.
	\newblock In O.~Cheong, K.-Y. Chwa, and K.~Park, editors, {\em Algorithms and
		Computation. ISAAC 2010. Lecture Notes in Computer Science, vol 6506}, pages
	194--205. Springer Berlin Heidelberg, 2010.
	
	\bibitem{Liu_zhang_2019}
	H.~Liu, R.~Zhang, and X.~Hu.
	\newblock Burning number of theta graphs.
	\newblock {\em Appl. Math. Comput.}, 361:246--257, 2019.
	
	\bibitem{Liu_2023}
	X.-C. Liu and X.~Yang.
	\newblock On the {T}ur\'an number of generalized theta graphs.
	\newblock {\em SIAM J. Discrete Math.}, 37(2):1237--1251, 2023.
	
	\bibitem{Loerinc_1978}
	B.~Loerinc.
	\newblock Chromatic uniqueness of the generalized {$\theta $}-graph.
	\newblock {\em Discrete Math.}, 23(3):313--316, 1978.
	
	\bibitem{KM_KVK_ric}
	K.~Mozhui and K.~V. Krishna.
	\newblock The representation number and the prn of stacked book graphs.
	\newblock In {\em Current Progress in Interdisciplinary Research, Select Papers
		of RIC 2024}, volume~3, pages 337--346. Springer Nature Singapore, 2025.
	
	\bibitem{Mozhui_Krishna_2025}
	K.~Mozhui and K.~V. Krishna.
	\newblock An upper bound for the permutation-representation number of bipartite
	graphs.
	\newblock {\em Journal of Information Processing}, 33:1033--1041, 2025.
	
	\bibitem{trotterbook}
	W.~T. Trotter.
	\newblock {\em Combinatorics and partially ordered sets: {D}imension theory}.
	\newblock Johns Hopkins Series in the Mathematical Sciences. Johns Hopkins
	University Press, Baltimore, MD, 1992.
	
	\bibitem{trotter_dim}
	W.~T. Trotter and J.~I. Moore.
	\newblock The dimension of planar posets.
	\newblock {\em J. Combinatorial Theory Ser. B}, 22(1):54--67, 1977.
	
	\bibitem{Trotter_2016}
	W.~T. Trotter and R.~Wang.
	\newblock Planar posets, dimension, breadth and the number of minimal elements.
	\newblock {\em Order}, 33(2):333--346, 2016.
	
	\bibitem{yanna82}
	M.~Yannakakis.
	\newblock The complexity of the partial order dimension problem.
	\newblock {\em SIAM J. Algebraic Discrete Methods}, 3(3):351--358, 1982.
	
	\bibitem{Zhai_2021}
	M.~Zhai, L.~Fang, and J.~Shu.
	\newblock On the {T}ur\'an number of theta graphs.
	\newblock {\em Graphs Combin.}, 37(6):2155--2165, 2021.
	
\end{thebibliography}

\appendix

\section{Planar Posets Corresponding to Melon Graphs}
\label{planar_poset_melon}

\begin{figure}[h!]
	\centering
	\begin{minipage}{.45\textwidth}
		\centering
		\begin{tikzpicture}[scale=0.4]
			
			% Vertices
			\vertex (z0)  at (-0.5,-1)  [label=below:$0'$] { };
			\vertex (a1)  at (1,1)   {};
			\vertex (a2)  at (2,-1)  {};
			\vertex (a3)  at (3,1)   {};
			\vertex (a4)  at (4,-1)   {};
			\vertex (a5) at (5,1) {};
			\vertex (z0p) at (6.5,-1)  [label=below:$0$] {};
			
			\vertex (b1)  at (1,2.5)   {};
			\vertex (b2)  at (2,0.5)  {};
			\vertex (b3)  at (3,2.5) [label=above:$\vdots$]  {};
			\vertex (b4)  at (4,0.5)   {};
			\vertex (b5) at (5,2.5) {};
			
			\vertex (c1)  at (1,4)   {};
			\vertex (c2)  at (2,2)  {};
			\vertex (c3)  at (3,4)   {};
			\vertex (c4)  at (4,2)   {};
			\vertex (c5) at (5,4) {};
			
			% Edges (zig-zag)
			\draw (z0) -- (a1) -- (a2) to node[midway,left] {$E_1$} (a3)--(a4)--(a5)--(z0p);
			\draw (z0) -- (b1) -- (b2)  to node[midway,left] {$E_2$} (b3)--(b4)--(b5)--(z0p);
			\draw (z0) -- (c1) -- (c2)  to node[midway,left] {$E_m$} (c3)--(c4)--(c5)--(z0p);
			
		\end{tikzpicture}
		\caption{Hasse diagram of a poset in Case-I.}
		\label{planar_poset_even}
	\end{minipage}\qquad
	\begin{minipage}{.45\textwidth}
		\centering
		\begin{tikzpicture}[scale=0.4]
			
			% Vertices
			\vertex (z0)  at (-0.3,-1)  [label=below:$0'$] { };
			\vertex (a1)  at (1,1)   {};
			\vertex (a2)  at (2,-1)  {};
			\vertex (a3)  at (3,1)   {};
			\vertex (a4)  at (4,-1)   {};
			
			\vertex (z0p) at (5.5,4)  [label=right:$0$] {};
			
			\vertex (b1)  at (1,2.5)   {};
			\vertex (b2)  at (2,0.5)  {};
			\vertex (b3)  at (3,2.5) [label=above:$\vdots$]  {};
			\vertex (b4)  at (4,0.5)   {};

			\vertex (c1)  at (1,4)   {};
			\vertex (c2)  at (2,2)  {};
			\vertex (c3)  at (3,4)   {};
			\vertex (c4)  at (4,2)   {};

			% Edges (zig-zag)
			\draw (z0) -- (a1) -- (a2)  to node[midway,left] {$E_1$} (a3)--(a4)-- (z0p);
			\draw (z0) -- (b1) -- (b2)  to node[midway,left] {$E_2$} (b3)--(b4) --(z0p);
			\draw (z0) -- (c1) -- (c2)  to node[midway,left] {$E_{m}$} (c3)--(c4) --(z0p);
			
		\end{tikzpicture}
		\caption{Hasse diagram of a poset in Case-II.}	
		\label{planar_poset_odd}	 
	\end{minipage}
	\begin{minipage}{.45\textwidth}
		\centering
		\begin{tikzpicture}[scale=0.4]
			
			% Vertices
			\vertex (z0)  at (-2,-1)  [label=below:$0'$] { };
			\vertex (a1)  at (1,0.5)   {};
			\vertex (a2)  at (2,-1)  {};
			\vertex (a3)  at (3,0.5)   {};
			\vertex (a4)  at (4,-1)   {};
			
			\vertex (z0p) at (5.3,6)  [label=above:$0$] {};
			
			\vertex (b1)  at (1,2)   {};
			\vertex (b2)  at (2,0.5)  {};
			\vertex (b3)  at (3,2) [label=above:$\vdots$]  {};
			\vertex (b4)  at (4,0.5)   {};

			\vertex (c1)  at (1,3.5)   {};
			\vertex (c2)  at (2,2)  {};
			\vertex (c3)  at (3,3.5)   {};
			\vertex (c4)  at (4,2)   {};

			% Edges (zig-zag)
			\draw (z0) -- (a1) -- (a2)  to node[midway,left] {$E_1$} (a3)--(a4)-- (z0p);
			\draw (z0) -- (b1) -- (b2)  to node[midway,left] {$E_2$} (b3)--(b4) --(z0p);
			\draw (z0) -- (c1) -- (c2)  to node[midway,below right=-15] {$E_{m-1}$} (c3)--(c4) --(z0p);
			\draw (z0) to [bend left] (z0p);
			\node (1) at (1, 4)  [label=above:$E_m$]  {};
			\draw ;
		\end{tikzpicture}
		\caption{Hasse diagram of a poset in Case-III.}
		\label{planar_poset_evenodd1}
	\end{minipage}\qquad
	\begin{minipage}{.45\textwidth}
		\centering
		\begin{tikzpicture}[scale=0.4]
			
			% Vertices
			\vertex (z0)  at (-2,-1)  [label=below:$0'$] { };
			\vertex (a1)  at (1,0.5)   {};
			\vertex (a2)  at (2,-1)  {};
			\vertex (a3)  at (3,0.5)   {};
			\vertex (a4)  at (4,-1)   {};
			
			\vertex (z0p) at (5.3,6)  [label=above:$0$] {};
			
			\vertex (b1)  at (1,2)   {};
			\vertex (b2)  at (2,0.5)  {};
			\vertex (b3)  at (3,2) [label=above:$\vdots$]  {};
			\vertex (b4)  at (4,0.5)   {};

			\vertex (c1)  at (1,3.5)   {};
			\vertex (c2)  at (2,2)  {};
			\vertex (c3)  at (3,3.5)   {};
			\vertex (c4)  at (4,2)   {};
			
			\vertex (d1)  at (-0.5,3.5)   {};
			\vertex (d2)  at (-1.5,3.5)  {};
			\vertex (d3)  at (-2.5,3.5)   {};

			% Edges (zig-zag)
			\draw (z0) -- (a1) -- (a2)  to node[midway,left] {$E_1$}  (a3)--(a4)-- (z0p);
			\draw (z0) -- (b1) -- (b2) to node[midway,left] {$E_2$} (b3)--(b4) --(z0p);
			\draw (z0) -- (c1) -- (c2) to node[midway, below right=-15] {$E_{t-1}$}(c3)--(c4) --(z0p);
			\draw (z0) to node[above right] {$E_{t}$} (d1) --(z0p);
			\draw (z0) to node[above ] {$\cdots$} (d2) --(z0p);
			\draw (z0) to node[above left] {$E_{m-1}$} (d3) --(z0p);
			\draw ;
		\end{tikzpicture}
		\caption{Hasse diagram of a poset in Case-IV.}
		\label{planar_poset_evenodd}		 
	\end{minipage}
\end{figure}

\begin{enumerate}
	\item 
	It is evident that the poset induced by a path is a planar poset. In fact, the vertices will be placed in two levels in its Hasse diagram. If a path is of even parity, then the end points will be in the same level and they will be in different levels for a path of odd parity.

	\item Extending this to melon graphs restricted to comparability graphs, it can be observed that the corresponding posets are planar posets. In view of Theorem \ref{chr_com_file}, for melon graphs, we provided a case wise depiction of the Hasse diagrams of the corresponding planar posets as per the following. 
	\begin{itemize}
		\item 
		Case-I: For the melon graphs in which all constituent paths are of even parity, see Fig. \ref{planar_poset_even}.  
		
		\item Case-II: For the melon graphs in which all constituent paths are of odd parity, see Fig. \ref{planar_poset_odd}. 
		
		\item Case-III: For the melon graphs in which the end points are adjacent and all constituent paths are of odd parity, see Fig. \ref{planar_poset_evenodd}. 
		
		\item Case-IV: For the melon graphs in which the end points are adjacent, see Fig. \ref{planar_poset_evenodd1}.  
	\end{itemize}
	
\end{enumerate}

\end{document}